\documentclass [12pt]{amsart}

\usepackage{graphicx}
\usepackage{amsmath}
\usepackage{amssymb}

\newtheorem{thm}{Theorem}
\newtheorem{Theorem}[thm]{Theorem}

\newtheorem{lemma}{Lemma}

\newtheorem{prop}{Proposition}

\newtheorem{rmk1}{Remark}

\newcommand{\M}{\mathcal{M}}
\newcommand{\Pf}{P_{1/2s}(\chi)}
\newcommand{\Mf}{M_{1/s}(\chi)}

\newcommand{\hf}{\tfrac{1}{2}}

\newcommand{\w}{\vec{w}}
\newcommand{\wone}{\vec{w_1}}
\newcommand{\wtwo}{\vec{w_2}}
\newcommand{\z}{\vec{z}}
\newcommand{\zone}{\vec{z_1}}
\newcommand{\ztwo}{\vec{z_2}}
\newcommand{\af}{\mathfrak{a}}
\newcommand{\Wf}{\mathcal{W}}

\newcommand{\bfrac}[2]{\left(\frac{#1}{#2}\right)}

\begin{document}

\title{Lower bounds for small fractional moments of Dirichlet $L$-functions}

\author{Vorrapan Chandee}
\email{vorrapan@buu.ac.th}
\address{Department of Mathematics, Burapha University, Chonburi, Thailand 20131; Centre de recherches math\'{e}matiques Universit\'{e} de Montr\'{e}al
P.O. Box 6128, Centre-ville Station Montr\'{e}al, Qu\'{e}bec
H3C 3J7}

\author{Xiannan Li}
\email{xiannan@illinois.edu}
\address{Department of Mathematics, University of Illinois at Urbana-Champaign, 1409 W. Green Street, Urbana, IL 61801 USA}


\date{\today}
\thanks
{The second author is partially supported by a NSERC PGS-D award.}
\maketitle
\begin{abstract}
We prove a lower bound of the correct order of magnitude in the conductor aspect for small rational moments of Drichlet $L$-functions.  Such bounds require new techniques, which is visible from the relationship to non-vanishing results for $L(1/2, \chi)$.
\end{abstract}

\section{Introduction}
Moments of $L$-functions on the critical line have been objects of intense scrutiny from analytic number theorists.  These moments are related to the size and value distribution of $L$-functions, and have applications towards both zero density and non-vanishing results.  The first moments studied were those of the Riemann zeta function, which are averages of the form
$$I_k(T) := \int_0^T |\zeta(\frac{1}{2} + it)|^{2k}dt.
$$  Here, asymptotic formulae were proven when $k = 1$ by Hardy and Littlewood and when $k=2$ by Ingham (see \cite{Ti} VII), and a well known folklore conjecture states that  $I_k(T) \sim c_k T(\log T)^{k^2}$ for constants $c_k$ depending on $k$.  This is supported by the random matrix model of Keathing and Snaith \cite{KS} from which explicit conjectures for $c_k$ were formulated.  In support of this conjecture, good upper bounds are available conditional on RH, due to Soundararajan, of the form $I_k(T) \ll T(\log T)^{k^2+\epsilon}$ \cite{Moment}.

These moments have also been studied when $k$ is not an integer.  Here, results are due to Conrey and Ghosh\cite{CG}, Heath-Brown \cite{HB1}, and Ramachandra \cite{Rama2}, \cite{Rama3}.  In particular, the results of Ramachandra state on RH that $I_k(T) \ll T(\log T)^{k^2}$ for $0\leq k\leq 2$ and $I_k(T) \gg T(\log T)^{k^2}$ for all real $k\geq 0$.  Heath-Brown \cite{HB1} proved the lower bound unconditionally for rational $k$ and the upper bound unconditionally for $k=1/n$ where $n \in \mathbb{N}$, using a convexity argument based on a result of Gaberiel \cite{Ga}.

Recently, Radziwill showed that $I_k(T)  \ll T(\log T)^{k^2}$ for all $k<2.181$, also conditionally on RH.  His work connects this question to the result of Hughes and Young \cite{HY} on the twisted fourth moments of $\zeta(s)$.  In a preprint, Radziwill and Soundararajan \cite{MS} prove unconditionally the lower bound of $I_k(T) \gg T(\log T)^{k^2}$ for all $k\geq 1$, and plan to do so for $k<1$.

Authors have also investigated analogous moments for families of $L$-functions in other aspects, and the simpliest example here are the Dirichlet $L$-functions.  To be more precise, we are interested in averages of the form 
$$\M_k(q) = \sum_{\chi\neq \chi_0} |L(1/2, \chi)|^{2k}
$$for real $k\geq 0$, where the sum is over all non-principal Dirichlet characters $\chi$ with modulus $q$.

In this direction, the method of Rudnick and Soundararajan \cite{RS} provides good lower bounds of the correct order of magnitude $\M_k(q) \gg \phi(q) \log^{k^2}q$ when $k\geq 1$ is rational and $q$ is prime.  Very recently, Heath-Brown \cite{D.R.HB} has shown that $\M_k(q) \ll \phi(q) \log^{k^2}q$ conditionally on GRH when $k\in (0, 2)$, and unconditionally when $k=1/v$ for $v$ a positive integer.  In the same paper \cite{D.R.HB}, Heath-Brown notes that the method of Rudnick and Soundararajan does not apply to find good lower bounds when $0<k< 1$ and that it would be interesting to derive such bounds.  The purpose of this paper is to fill this gap in the literature.  More precisely, we prove the following.

\begin{Theorem}  \label{thm:lowerfractional}
For $\M_k(q) = \sum_{\chi\neq \chi_0} |L(1/2, \chi)|^{2k}$, we have that
$$\M_k(q) \gg_k \phi(q) \log^{k^2}q
$$ for all rational $k \in (0, 1)$ and prime $q$.
\end{Theorem} 

For example, combined with Heath-Brown's work \cite{HB1}, this gives us the following unconditional estimate for the average size of $L(1/2, \chi)$:
$$\frac{1}{\phi(q)} \sum_{\chi\neq \chi_0} |L(1/2, \chi)| \asymp (\log q)^{1/4}.
$$

  The proof of the theorem should generalize to prove the analogous result for any family of $L$-functions for which we can take a twisted second moment.

\begin{rmk1}
If the result above holds with sufficient uniformity with respect to $k$, we would have that
$$\# \{\chi \textup{ mod } q: L(1/2, \chi)\neq 0\} = \lim_{k\rightarrow 0^+} \M_k(q) \gg \phi(q),
$$thus proving that a positive proportion of such $L$-functions do not vanish at the critical point.  Thus, it should come as no surprise that mollifiers play a role in our proof, even though mollifiers do not appear in the method of Rudnick and Soundararajan \cite{RS}.  We remind the reader that the now standard approach to proving such non-vanishing results relies heavily on mollification.  See \cite{Snonvanishing} for example.  \footnote{We refrain from attempting to prove our result with such uniformity as this would be an overly complicated method to prove non-vanishing compared to the standard mollifying argument.  We mention this connection mainly to point out this special feature of small moments.}  
\end{rmk1}

The proof begins with the same approach as the method of Rudnick and Soundararajan in \cite{RS}.  As mentioned in the above remark, we introduce short mollifiers which allows us to study smaller moments than before.  However, it also introduces some initially unexpected subtleties into the problem.

A careful choice reduces the problem to evaluating certain moments where the highest power of $L(s, \chi)$ is $2$.  These moments can be evaluated by taking the diagonal terms, which reduces the problem to evaluating the integrals (\ref{integral_lower_bound}) and (\ref{integral_upp_bound}).  This turns out to be a delicate task.  In general, it is onerous to asymptotically evaluate even specific instances of these integrals - see \cite{KMV} for the asymptotic evaluation of a simpler integral.  Previous work has dealt with integrals of the same shape, but which either contain only integer powers of $\zeta$ (in which case the integrand is meromorphic) or positive powers of $\zeta$ (in which case the coefficients of the Dirichlet series are positive).  These cases are substantially simpler than the integrals which appear in our method.  However, our task is made somewhat simpler by the fact that we only require the correct order of magnitude for these integrals.  

In the next section, we outline the proof of the theorem and explain the new features of the problem in more detail. 

\section{An outline of the proof}
Let $k = \frac{r}{s},$ for coprime positive integers $r$ and $s$, and for a positive rational number $\alpha,$ define
$$ P_{\alpha}(\chi) :=  \sum_{n \leq x} \frac{d_\alpha(n) \chi(n)}{\sqrt{n}}\frac{\log \frac{x}{n}}{\log x},$$
and 
$$ M_{\alpha}(\chi) :=  \frac{1}{2}\sum_{n \leq y} \frac{d_\alpha(n) \mu(n)\chi(n)}{\sqrt n} \frac{\log^2 \frac{y}{n}}{\log^2 y},$$
where as usual, $d_r(n)$ denotes the coefficients of $\zeta(s)^r$ as defined in Chapter II.5 in \cite{Ten}.  
The reader should think of $x$ and $y$ as small powers of $q$.  On average, we shall expect $P_\alpha(\chi)$ to behave like $L(1/2, \chi)^\alpha$, and the mollifier $M_\alpha(\chi)$ to behave like $L(1/2, \chi)^{-\alpha}$.  We let $x = y^a$, where $a$ is a large fixed positive constant to be determined later and which depends only on $r$ and $s$, and further demand that $x^{4s} \leq q^{1/20}$.
\begin{rmk1} \label{remS} Since $0 < k < 1,$ $s \geq 2.$ This fact will be used later in the proof. 
\end{rmk1}
Define 
$$ S_l := \sum_{\chi \neq \chi_0}  L( \hf, \chi) \overline{\Pf}^{2s} |\Mf|^{2(s - r)} $$
and 
$$ S_u := \sum_{\chi \neq \chi_0} \left|L\left( \hf, \chi \right)\right|^{2} |\Pf|^{4s} |\Mf|^{2(2s - r)}.$$
By H\"older's inequality and Cauchy's inequality, 
\begin{align*}
|S_l| &\leq \left( \sum_{\chi \neq \chi_0}  \left|L\left( \hf, \chi \right)\right|^{k} |\Pf|^{2r} \right)^{\frac{1}{2 - k}} (S_u)^{\frac{1 - k}{2 - k}} \\
&\leq \left(\sum_{\chi \neq \chi_0} \left|L\left( \hf, \chi \right)\right|^{2k}\right)^{\frac{1}{2(2-k)}}\left(\sum_{\chi \neq \chi_0} |\Pf|^{4r}\right)^{\frac{1}{2(2-k)}} (S_u)^{\frac{1 - k}{2 - k}}.
\end{align*}
As in the method of Rudnick and Soundararajan, it is important that the sums appearing above should all have the same size, since then H\"older's and Cauchy's inequalities are essentially sharp. Indeed, we will show that $S_u \ll \phi(q)(\log q)^{k^2} \ll S_l$. The theorem will then follow.  

\begin{rmk1}
We have motivated the appearance of mollifiers, but are the $P_{1/2s}(\chi)$ necessary?  For instance, why not write $|L(1/2, \chi)|^2$ in $S_l$ instead of $L(1/2, \chi) \overline{P_{1/2s}(\chi)}^{2s}$?  Having $|L(1/2, \chi)|^2$ in $S_l$ causes a factor of $|L(1/2, \chi)|^4$ to appear in $S_u$.  This leads to the problem of evaluating a twisted fourth moment, which is significantly more challenging to evaluate.  Recently, M. Young \cite{Y} has evaluated the fourth moment with a power saving error term.
\end{rmk1}

For any positive integers $\mathcal{A}, \mathcal{B} , n$ we define
\begin{equation} \label{def:dr/s}
d_{\mathcal{A}/\mathcal{B}}(n, x) := \sum_{\substack{n_1n_2...n_{\mathcal{A}} = n \\ n_i \leq x}} d_{1/\mathcal{B}}(n_1)...d_{1/\mathcal{B}}(n_{\mathcal{A}})\frac{\log \tfrac{x}{n_1}}{\log x}...\frac{\log \tfrac{x}{n_{\mathcal{A}}}}{\log x}.
\end{equation}
Then
\begin{align*}
\sum_{\chi \neq \chi_0} |\Pf|^{4r} &= \sum_{\chi \neq \chi_0} \sum_{n,m \leq x^{2r}} \frac{d_{2r/2s}(n, x) d_{2r/2s}(m, x) \chi (n) \overline{\chi} (m)}{\sqrt{nm}} \\
& \ll \phi(q) \sum_{n \leq x^{2r}} \frac{d_{2r/2s}^2(n, x)}{n} \\
& \ll \phi(q) \sum_{n \leq x^{2r}} \frac{d_{r/s}^2(n)}{n} \ll \phi(q)(\log q)^{r^2/s^2},
\end{align*}
Note that the last inequality follows from Lemma 2 in \cite{HB1}.

To bound $S_l$ and $S_u$, we take advantage of the fact that both $\Pf$ and the mollifier $\Mf$ are short, which allows us to apply orthogonality.  This will convert the sums into contour integrals as in the following propositions.

\begin{prop} \label{prop:integralLower}Let $\wone = (w_1,...,w_{s})$, $\wtwo = (w_{s + 1},...,w_{2s})$, and $\w = (w_1,...,w_{2s})$.  Similarly define $\zone = (z_1,...,z_{s - r})$, $\ztwo = (z_{s - r + 1},...,z_{2(s - r)})$, and $\z = (z_1,...,z_{2(s - r)})$.  We will write symbols like $d\vec w$ with the obvious interpretation. Then  
$$S_l = \frac{\phi(q) I_l}{(\log x)^{2s}(\log y)^{4s-4r}} + O\bfrac{q}{\log q},$$ where for $c>0$,
\begin{align} \label{integral_lower_bound} 
&I_l =  \frac{1}{(2\pi i)^{4s-2r + 1} }\int_{(c)^{4s - 2r + 1}} \frac{\prod_{1 \leq i \leq 2s} \zeta^{1/2s}(1 + w_0 + w_{i}) \prod_{1 \leq  j, k \leq s - r} \zeta^{1/s^2}(1 + z_j + z_{s - r + k})}{ \prod_{\substack{1 \leq i \leq 2s \\ 1 \leq j \leq s - r}} \zeta^{1/2s^2}(1 + w_{i} + z_{j}) \prod_{1 \leq k \leq s - r} \zeta^{1/s}(1 + w_0 + z_{s - r + k})}   \\
& \ \ \ \ \ \ \ \ \ \ \ \ \ \ \ \ \ \ \ \ \ \ \ \ \ \ \ \ \ \ \ \ \ \ \ \ \ \ \ \ \ \ \  \cdot X^{w_0}\Gamma(w_0) \frac{x^{w_1 + ... +w_{2s}}y^{ z_1 + ... + z_{2(s - r)}}}{w_1^2...w_{2s}^2z_1^3...z_{2(s-r)}^3}  \eta_l(w_0, \vec{w}, \vec{z}) \> dw_0 \> d\vec{w} \> d\vec{z}, \nonumber
\end{align}
where $\eta_l(w_0, \vec{w}, \vec{z})$ is an absolutely convergent Euler product for $w_i, z_i$ in the domain ${\rm Re}(w_i), {\rm Re}(z_i) \geq -3/16.$ 
\end{prop}

\begin{prop} \label{prop:integralUpper} We define $\wone = (w_1,...,w_{2s})$, $\wtwo = (w_{2s + 1},...,w_{4s})$, and $\w = (w_1,...,w_{4s})$.  Similarly define $\zone = (z_1,...,z_{2s - r})$, $\ztwo = (z_{2s - r + 1},...,z_{2(2s - r)})$, and $\z = (z_1,...,z_{2(2s - r)})$.  Then 
$$ S_u \ll \frac{\phi(q) I_u}{(\log x)^{4s}(\log y)^{8s - 4r}} + \frac{q}{\log q},$$
where for $c>0$,
\begin{align} \label{integral_upp_bound}
& I_u = \frac{1}{\Gamma\left( \tfrac{1}{4}\right)^2 } \frac{1}{(2\pi i)^{8s - 2r + 1}}\int_{(c)^{8s - 2r + 1}} \frac{\prod_{1 \leq j \leq 4s} \zeta^{1/2s}(1 + w_0 + w_j)}{\prod_{1 \leq j \leq 4s - 2r} \zeta^{1/s}(1 + w_0 + z_{j})}\\
& \hspace{1in}\cdot \frac{\prod_{1 \leq i, j \leq 2s } \zeta^{1/4s^2} (1 + w_i + w_{2s + j}) \prod_{1 \leq i, j \leq 2s-r} \zeta^{1/s^2} (1 + z_i + z_{2s - r + j}) }{\prod_{\substack{1 \leq i \leq 2s \\ 1 \leq j \leq 2s - r}} \zeta^{1/2s^2}(1 + w_i + z_{2s - r + j})\zeta^{1/2s^2}(1 + w_{2s + i} + z_{j}) }  \nonumber \\
& \hspace{1.2in} \cdot \Gamma\left( \frac{1}{4} + \frac{w_0}{2}\right)^2 \zeta(1 + 2w_0) \left( \frac{q}{\pi}\right)^{w_0} \frac{x^{w_1 + ... + w_{4s}}y^{z_1 + ... + z_{4s - 2r}}}{w_1^2...w_{4s}^2z_1^3...z_{4s - 2r}^3} \eta_u(w_0, \vec{w}, \vec{z}) \frac{dw_0}{w_0} \> d\w \> d\z \nonumber
\end{align} 
where $\eta_{u}(w_0, \vec{w}, \vec{z})$ is an absolutely convergent Euler product for $w_i, z_i$ in the domain ${\rm Re}(w_i) , {\rm Re}(z_i)\geq -3/16.$ 
\end{prop} 
We now give a short description of how one can write down the essential parts of the integrals from the sums directly.  Take the sum $S_l$ for example.  Associate $w_0$ with $L(1/2, \chi)$, $w_i$ with $\overline{\Pf}$ for $1\leq i\leq 2s$, $z_j$ with $\Mf$ for $1\leq j\leq s-r$ and $z_{s-r+k}$ with $\overline{\Mf}$ for $1\leq k\leq s-r$.  Then pair each factor containing $\chi$ with a factor containing $\overline{\chi}$ (this comes from applying orthogonality).  Pairing $L(1/2, \chi)$ with $\overline{\Pf}$ gives the factors $\zeta^{1/2s}(1+w_0+w_i)$ for $1\leq i\leq 2s$, pairing $L(1/2, \chi)$ with $\overline{\Mf}$ gives the factors $\zeta^{-1/s}(1+w_0+z_{s-r+k})$ for $1\leq k\leq s-r$, pairing $\Mf$ with $\overline{\Pf}$ gives $\zeta^{-1/2s^2}(1+w_i+z_j)$ for $1\leq i\leq 2s$ and $1\leq j\leq s-r$, and pairing $\Mf$ with $\overline{\Mf}$ gives $\zeta^{1/s^2}(1 + z_j + z_{s-r+k})$ for $1 \leq j, k \leq s-r.$  The detailed proofs of Proposition \ref{prop:integralLower} and \ref{prop:integralUpper} reside in \S \ref{sec:contourLower} and \S \ref{sec:contourUpper}, respectively.

The Propositions above have reduced the problem to one of bounding $I_u$ and $I_l$.  We state the bounds in the following proposition.
\begin{prop} \label{prop:boundforIuIl} Let $I_u$ and $I_l$ be as in Proposition 
\ref{prop:integralUpper} and \ref{prop:integralLower}. Then 
\begin{equation} \label{prop:upperbound}
I_u \ll (\log q)^{\frac{r^2}{s^2} + 12s - 4r},
\end{equation}
and 
\begin{equation} \label{prop:lowerbound}
I_l \gg (\log x)^{2s+1}(\log y)^{\frac{r^2}{s^2} + 4s-4r - 1},
\end{equation}
where the implied constant may depend on $r$ and $s.$
\end{prop}
From Proposition \ref{prop:boundforIuIl}, we deduce immediately that $S_u \ll (\log q)^{r^2/s^2} \ll S_l$ since $x$ and $y$ are powers of $q$, which proves the main result.  We now describe a short heuristic for determining the size of the integrals.  Examine $I_l$, and add up the total order of singularities in the integrand.  The $\zeta^{1/2s}(1+w_0+w_i)$ terms contribute $\frac{1}{2s} 2s =1$, the $\zeta^{-1/s}(1+w_0+z_{s-r+k})$ gives $-\frac{1}{s} (s-r) = -\frac{s-r}{s}$, the $\zeta^{-1/2s^2}(1+w_i+z_j)$ terms gives $-\frac{1}{2s^2} (2s)(s-r) = -\frac{s-r}{s}$  and the $\zeta^{1/s^2}(1 + z_j + z_{s-r+k})$ terms contribute $\frac{1}{s^2}(s-r)^2.$  Moreover, $\Gamma(w_0)$, $\frac{1}{w_i^2}$, and $\frac{1}{z_j^3}$ gives $1+4s+6(s-r)$.  Adding these up and taking away the number of variables gives $\frac{r^2}{s^2}+2s+4(s-r)$.  Note that we expect $I_l \asymp (\log q)^{\frac{r^2}{s^2}+2s+4(s-r)}$.

The upper bound for $I_u$ is relatively straightforward, and proven in \S \ref{sec:evaluateIu}.  Rigorously proving the lower bound for $I_l$ is quite delicate.  Evaluating these integrals typically consists of considering many sequences of singularities which all contribute terms of the same size.  We refer the reader to \cite{KMV} for an actual example which appear in applications.  For now, consider the simple example of $\int_{(c)^2} \zeta^\alpha(1+w+z) \frac{x^wy^z}{w^2z^3} dw dz$.  Shifting the contour in $w$ to the left gives two singularities, one at $w=0$ and one at $w=-z$.  Then shifting $z$ to the left gives the singularity $z=0$.  This gives two sequences of singularities to consider, namely $w=0, z=0$ and $w=-z$ and $z=0$.  In $I_l$, the number of such sequences to consider grows arbitrarily large as $r$ and $s$ grow.  In order to prove a lower bound, one would then need to show that the contribution from these sequences do not cancel out.  We avoid this complexity by making the mollifier $\Mf$ much shorter\footnote{This is done through our choice $x = y^a$.} compared to the Dirichlet polynomials $\Pf$.  This has the effect of making the $w_0=0$, $w_i=0$, $z_i=0$ sequence the dominant one, in the sense that the coefficients in front of the contributions from other sequences are much smaller, depending on $a$.  This is carried out in \S \ref{sec:evaluateIl}.

\section{The contour integral for $S_l$} \label{sec:contourLower}
In this section, we will prove Proposition \ref{prop:integralLower}. First we will write $L(1/2, \chi)$ in terms of Dirichlet series by applying a standard technique with Perron's formula to 
$$\frac{1}{2\pi i} \int_{(1)} L\left(s + \tfrac{1}{2}, \chi\right)\Gamma(s)X^sds.$$ 
We set $X = q^{5/4}$, and see from the calculus of residues that

\begin{equation} \label{eqn:approxforL}
 L(\hf, \chi) = \sum_{m \leq X^{1+\epsilon}} \frac{\chi(m)}{\sqrt m} e^{-m/X} + O\left(\frac{\log q}{q^{1/8}}\right),
\end{equation}
where we have shifted contours to $\Re s = -1/2$ and using the bound $L(it, \chi) \ll_t q^{1/2}\log q$.
 
Therefore
\begin{align} \label{S_1}
 S_l = \sum_{\chi} \sum_{m \leq X^{1+\epsilon}} \frac{\chi(m)}{\sqrt m} e^{-m/X}& \overline{\Pf}^{2s} |\Mf|^{2(s - r)}  \\
& + O\left( \frac{\log q}{q^{1/8}}\sum_{\chi } |\Pf|^{2s} |\Mf|^{2(s - r)} \right). \nonumber
\end{align}
We can write 
$$ |\Pf|^{2s} |\Mf|^{2(s - r)} = \sum_{m, n \leq x^sy^{s-r}} \frac{c_mc_n \chi(m)\bar{\chi}(n)}{\sqrt{mn}},$$
where $c_n \ll_{\epsilon} n^{\epsilon}.$ Since $x^sy^{s-r} \leq x^{4s} \leq q^{1/10},$ applying orthogonality gives
$$ \frac{\log q}{q^{1/8}} \sum_{\chi} |\Pf|^{2s} |\Mf|^{2(s - r)} \ll  q^{9/10} \sum_{n \leq x^sy^{s - r}} \frac{c_n^2}{n} \ll \frac{q}{\log q}.$$

Now we consider the main term, which is
\begin{align*}
& \sum_{\chi} \sum_{m \leq X^{1+\epsilon} } \frac{\chi(m)}{\sqrt m } e^{-m/X} \overline{\Pf}^{2s} |\Mf|^{2(s - r)} \\
&= \phi(q) \sum_{n \leq x^{2s}} \sum_{a, b \leq y^{s - r}} \sum_{\substack{m \leq X^{1+\epsilon} \\ ma \equiv nb \ ({\rm mod} \ q)} }\frac{e^{-m/X} d_{2s/2s} (n,x) d^*_{(s-r)/s}(a,y) d^*_{(s-r)/s}(b,y)}{\sqrt{mnab}},
\end{align*}
where $d_{2s/2s}(n, x)$ is defined as in (\ref{def:dr/s}), and for any positive integers $\mathcal{A}, \mathcal{B}, n$ we define
$$ d^*_{\mathcal{A}/\mathcal{B}}(n, y) = \frac{1}{2^{\mathcal{A}}}\sum_{\substack{n_1...n_{\mathcal{A}} = n \\ n_i \leq y}} d_{1/\mathcal{B}}(n_1)...d_{1/\mathcal{B}}(n_{\mathcal{A}})\mu(n_1)...\mu(n_{\mathcal{A}})\frac{\log^2 \frac{y}{n_1}}{\log^2 y}...\frac{\log^2 \frac{y}{n_{\mathcal{A}}}}{\log^2 y},$$
The main term will arise from the diagonal term $ma = nb.$ First we will compute the contribution of the off-diagonal terms. Without loss of generality we may write $ma = nb + ql,$ where $1 \leq l \leq y^{s-r}X^{1+\epsilon}/q.$ 

Since $x^{s}y^{s-r} \leq q^{1/10}$, $d_{2s/2s} (n,x) \ll_{\epsilon} n^{\epsilon}$, $d^*_{(s-r)/s}(b,y) \ll_{\epsilon} b^{\epsilon}$ and 
$$\sum_{ma = nb + ql} | d^*_{(s-r)/s}(a,y)| \leq \frac{1}{2^{s-r}}\sum_{ma = nb + ql} d_{(s - r)/s}(a) \ll d_{\frac{s-r}{s} + 1}(nb + ql) \ll (nb + ql)^{\epsilon} \ll q^{\epsilon},$$ the contribution of the off-diagonal terms is 
\begin{align} \label{errorinlowerbound}
\ll \phi(q) \sum_{n \leq x^{2s}} \sum_{ b \leq y^{s - r}} \frac{d_{2s/2s} (n,x)  |d^*_{(s-r)/s}(b,y)|}{\sqrt{nb}} \sum_{l \leq y^{s-r}q^{1/4 + \epsilon}} \frac{1}{\sqrt{ql}} \sum_{ma = nb + ql} |d^*_{(s-r)/s}(a,y)| \ll \frac{q}{\log q}.
\end{align}
We will now consider the diagonal term. In the sequel, $\int_{(c)^{l}}$ shall denote a $l$-times interated integral on the vertical line real part $= c.$ 
By Perron's formula, for $c > 1/2,$ we write
\begin{equation} \label{eqn:Pf}
 \Pf = \frac{1}{\log x}\frac{1}{2\pi i} \int_{(c)} \sum_{m} \frac{d_{1/2s}(m)\chi(m)}{\sqrt m} \left(\frac{x}{m}\right)^w \> \frac{dw}{w^2}
\end{equation}
and 
\begin{equation} \label{eqn:Mf}
\Mf =  \frac{1}{\log^2 y} \frac{1}{2\pi i} \int_{(c)} \sum_{a} \frac{d_{1/s}(a)\mu(s)\chi(a)}{\sqrt a} \left(\frac{y}{a}\right)^z \> \frac{dz}{z^3}.
\end{equation}
 Therefore 
\begin{equation} \label{eqn:Pf_2s}
\Pf^{s} = \frac{1}{(\log x)^s}\frac{1}{(2\pi i)^s} \int_{(c)^{s}} \sum_m \frac{\chi(m)}{\sqrt{m}} \sigma_{\wone} (m) \frac{x^{w_1 + ... + w_{s}}}{w_1^2...w_{s}^2} \> d\wone 
\end{equation}
and 
\begin{equation}\label{eqn:Mf_SminusR}
\Mf^{s -r} = \frac{1}{(\log y)^{2s-2r}} \frac{1}{(2\pi i)^{s-r}} \int_{(c)^{s - r}} \sum_a \frac{\chi(a)}{\sqrt{a}} \rho_{\zone} (a) \frac{y^{z_1 + ... + z_{s - r}}}{z_1^3...z_{s-r}^3} \> d\zone, 
\end{equation}
where for any $k$, we define
\begin{align} \label{def:SigmaAndDelta}
\sigma_{w_1,...,w_k} (n) &= \sum_{n_1...n_k = n} \frac{d_{1/2s}(n_1)...d_{1/2s}(n_k)}{n_1^{w_1}...n_s^{w_k}}, \hspace{1in} \textup{and}\\
\rho_{z_1,...,z_k} (a) &= \sum_{a_1...a_k = a} \frac{d_{1/s}(a_1)...d_{1/s}(a_k)\mu(a_1)...\mu(a_k)}{a_1^{z_1}...a_k^{z_k}}. \nonumber 
\end{align}
From (\ref{eqn:Pf_2s}) and (\ref{eqn:Mf_SminusR}) and the fact that $x^{2s} < q,$ the diagonal term is
\begin{align*}
& \frac{1}{(\log x)^{2s}(\log y)^{4s-4r}}\frac{\phi(q)}{(2\pi i)^{4s - 2r + 1}} \int_{(c)^{4s - 2r + 1}}  \sum_{ma = nb} \frac{1}{\sqrt{mnab}} \frac{1}{m^{w_0}} \sigma_{\w} (n)   \rho_{\zone} (a) \rho_{\ztwo} (b)  \\
& \ \ \ \ \ \ \ \ \ \ \ \ \ \ \ \ \ \ \ \ \ \ \ \ \ \ \ \ \ \ \ \ \ \ \ \ \ \ \ \ \ \ \ \ \ \ \ \ \ \ \ \cdot X^{w_0}\Gamma(w_0)\frac{x^{w_1 + ...+ w_{2s}} y^{ z_1 + ... + z_{2(s - r)}}}{w_1^2...w_{2s}^2z_1^3...z_{2(s-r)}^3} \>dw_0 \>d\vec{w} \>d\vec{z}
\end{align*}
where $d\vec{w} = \> dw_1 ...dw_{2s}$ and $d\vec{z} = \> dz_1 ...dz_{2(s - r)}.$  Since $ma = nb,$ there exists $\alpha$ such that $m = \frac{\alpha b}{(a,b)}$ and $n = \frac{\alpha a}{(a,b)}$.  Thus the above expression equals
\begin{align*}
& \frac{1}{(\log x)^{2s}(\log y)^{4s-4r}}\frac{\phi(q)}{(2\pi i)^{4s - 2r + 1}}  \int_{(c)^{4s - 2r + 1}} \sum_{a,b} \frac{(a,b)^{1 + w_0}}{ab^{1 + w_0}} \rho_{\zone} (a) \rho_{\ztwo} (b) \sum_{\alpha} \frac{\sigma_{\w} \left(\frac{\alpha a}{(a,b)}\right)}{\alpha^{1 + w_0}} \\
& \ \ \ \ \ \ \ \ \ \ \ \ \ \ \ \ \ \ \ \ \ \ \ \ \ \ \ \ \ \ \ \ \ \ \ \ \ \ \ \ \ \ \ \ \ \ \ \ \ \ \ \cdot X^{w_0}\Gamma(w_0)\frac{x^{w_1 + ... + w_{2s}}y^{ z_1 + ... + z_{2(s - r)}}}{w_1^2...w_{2s}^2z_1^3...z_{2(s-r)}^3}  \>dw_0 \> d\vec{w} \> d\vec{z},
\end{align*}
 Let $h = \frac{a}{(a,b)}$. Since $\sigma_{\vec{w}} (n)$ is multiplicative in $n$, we write the sum over $\alpha$ as an Euler product of the form 
$$A_{\w} \cdot B_{\w}(h).$$
Here,
$$A_{\w}  =  \prod_{p}  \left(\sum_{j = 0}^{\infty} \frac{\sigma_{ \w} (p^j)}{p^{j(1 + w_0)} }\right) = \prod_{1 \leq i \leq 2s} \zeta^{1/2s}(1 + w_0 + w_{i})\eta(w_0, \vec{w}), $$
where $\eta(w_0, \vec{w})$ is absolutely convergent and bounded for $w_i$ in the
domain ${\rm Re}(w_i) \geq -3/16, $ and 
\begin{align*}
 B_{\w}(h) &=  \prod_{p | h } \frac{\frac{\sum_{j = 0}^{\infty} \sigma_{\w}(p^{j + h_p})}{p^{j(1 + w_0)}}}{{\sum_{j = 0}^{\infty} \frac{\sigma_{\w}(p^j)}{p^{j(1 + w_0)}}}},
\end{align*}
where $h_p$ is the highest power of $p$ dividing $h.$
Hence the diagonal term is
\begin{align*}
&\frac{1}{(\log x)^{2s}(\log y)^{4s-4r}}\frac{\phi(q)}{(2\pi i)^{4s - 2r + 1}}  \int_{(c)^{4s - 2r + 1}}  \sum_{a, b} \frac{(a,b)^{1 + w_0}}{ab^{1 + w_0}} \rho_{\zone} (a) \rho_{\ztwo} (b)  A_{\w} \cdot B_{\w}\left(\frac{a}{(a,b)}\right) \\
& \ \ \ \ \ \ \ \ \ \ \ \ \ \ \ \ \ \ \ \ \ \ \ \ \ \ \ \ \ \ \ \  \ \ \ \ \ \ \ \ \ \ \ \ \ \ \ \ \ \ \ \cdot  X^{w_0}\Gamma(w_0)\frac{x^{w_1 + ...+ w_{2s}}y^{ z_1 + ... + z_{2(s - r)}}}{w_1^2...w_{2s}^2z_1^3...z_{2(s-r)}^3} \> dw_0 \> d\vec{w} \> d\vec{z}\\
&=  \frac{1}{(\log x)^{2s}(\log y)^{4s-4r}}\frac{\phi(q)}{(2\pi i)^{4s - 2r + 1}}   \int_{(c)^{4s - 2r + 1}}  A_{\w} \sum_{\substack{a, b \\ (a,b) = 1}}  \frac{B_{\w}(a)}{ab^{1 + w_0}} \sum_{d}  \frac{\rho_{\zone} (ad) \rho_{\ztwo} (bd)}{d}   \\
& \ \ \ \ \ \ \ \ \ \ \ \ \ \ \ \ \ \ \ \ \ \ \ \ \ \ \ \ \ \ \ \ \ \ \ \ \  \ \ \ \ \ \ \ \  \ \ \ \ \ \ \ \ \cdot X^{w_0}\Gamma(w_0)\frac{x^{w_1 + ... +w_{2s}}y^{ z_1 + ... + z_{2(s - r)}}}{w_1^2...w_{2s}^2z_1^3...z_{2(s-r)}^3}  \> dw_0 \> d\vec{w} \> d\vec{z}
\end{align*}
Since $\rho_{\zone} (n)$ is multiplicative and $(a,b) = 1$, the Euler product of the sum over $d$ is
$$ C_{\z} \cdot D_{\z}(a,1) \cdot  D_{\z}(1,b),$$
where 
$$C_{\z}  =   \prod_{p} \left(\sum_{j = 0}^{\infty} \frac{\rho_{\zone} (p^j)\rho_{\ztwo}(p^j)}{p^j}\right) = \prod_{1 \leq j, k \leq s - r} \zeta^{1/s^2}(1 + z_j + z_{s - r + k})\eta_1(\vec{z}), $$
where $\eta_1(\vec{z})$ is absolutely convergent and bounded for $z_i$ in the
domain ${\rm Re}(z_i) \geq -3/16, $ and 
\begin{align*}
 D_{\z}(a, 1) &= \prod_{p | a }  \left(\frac{\sum_{j = 0}^{\infty} \frac{\rho_{\zone} (p^{j+ a_p})\rho_{\ztwo}(p^{j})}{p^j}}{\sum_{j = 0}^{\infty} \frac{\rho_{\zone} (p^j)\rho_{\ztwo}(p^j)}{p^j}}\right) \hspace{0.1in} \textup{and} \\
D_{\z}(1, b) &=  \prod_{p | b } \left( \frac{\sum_{j = 0}^{\infty} \frac{\rho_{\zone} (p^{j})\rho_{\ztwo}(p^{j + b_p})}{p^j}}{\sum_{j = 0}^{\infty} \frac{\rho_{\zone} (p^j)\rho_{\ztwo}(p^j)}{p^j}}\right),
\end{align*}
where $a_p$ and $b_p$ are the highest power of $p$ dividing $a$ and $b$ respectively. 

Now we have that the diagonal term is 
\begin{align*}
&= \frac{1}{(\log x)^{2s}(\log y)^{4s-4r}}\frac{\phi(q)}{(2\pi i)^{4s - 2r + 1}} \int_{(c)^{4s - 2r + 1}} A_{\w} C_{\z}  \sum_{\substack{a, b \\ (a,b) = 1}}  \frac{B_{\w}(a)D_{\z}(a,1) D_{\z}(1,b)}{ab^{1 + w_0}}   \\
& \ \ \ \ \ \ \ \ \ \ \ \ \ \ \ \ \ \ \ \ \ \ \ \ \ \ \ \ \ \ \ \ \ \ \ \ \  \ \ \ \ \ \ \ \  \ \ \ \ \ \ \ \ \cdot  X^{w_0}\Gamma(w_0) \frac{x^{w_1 + ... +w_{2s}}y^{ z_1 + ... + z_{2(s - r)}}}{w_1^2...w_{2s}^2z_1^3...z_{2(s-r)}^3}  \> dw_0 \> d\vec{w} \> d\vec{z}
\end{align*}
We will now consider the sums over $a$ and $b.$
\begin{align*}
&\sum_{\substack{a, b \\ (a,b) = 1}}  \frac{B_{\w}(a)D_{\z}(a,1) D_{\z}(1,b)}{ab^{1 + w_0}} 
= \sum_a  \frac{B_{\w}(a)D_{\z}(a,1) }{a} \sum_{\substack{b \\ (a,b) = 1}}  \frac{D_{\z}(1,b)}{b^{1 + w_0}} \\
&= \sum_a  \frac{B_{\w}(a)D_{\z}(a,1) }{a} \frac{1}{\prod_{p|a} 1 + \frac{1}{p} \eta_3(w_0, \vec{w}, \vec{z}, p)} 
\prod_{1 \leq k \leq s - r} \zeta^{-1/s}(1 + w_0 + z_{s - r + k})\eta_4(w_0, \vec{w}, \vec{z}) \\
&= \prod_{\substack{1 \leq i \leq 2s \\ 1 \leq j \leq s - r}} \zeta^{-1/2s^2}(1 + w_{i} + z_{j}) \prod_{1 \leq k \leq s - r} \zeta^{-1/s}(1 + w_0 + z_{s - r + k})\eta_5(w_0, \vec{w}, \vec{z}),
\end{align*}
where each $\eta_i(w_0, \vec{w}, \vec{z})$ appearing above is given by an Euler product with real coefficients which is absolutely convergent for $w_i, z_i$ in the domain ${\rm Re}(z_i), {\rm Re}(w_i) \geq -3/16$.

Combining all computations above, we obtain that the diagonal term is 
$\frac{\phi(q) I_l}{(\log x)^{2s}(\log y)^{4s-4r}}$. This proves Proposition \ref{prop:integralLower}.

\section{The contour integral for $S_u$} \label{sec:contourUpper}  
In this section, we will prove Proposition \ref{prop:integralUpper}. 
 Recall that $S_u$ is defined to be
$$ \sum_{\chi \neq \chi_0} \left|L\left( \hf, \chi \right)\right|^{2} |\Pf|^{4s} |\Mf|^{2(2s - r)}.$$
We first start by writing the approximate functional equation for $|L(1/2, \chi)|^2$. Let $\xi(s, \chi)$ be a completed $L$-function,
$$ \xi(s,\chi) = \left(\frac{q}{\pi}\right)^{(s + \af)/2} \Gamma\left( \frac{s + \af}{2}\right) L(s,\chi),$$
where 
$$ \af = \left\{ \begin{array}{ll} 0 & {\rm if} \,\,\,\, \chi(-1) = 1 \\
1 & {\rm otherwise.}\end{array} \right.$$
For $c > 1/2$ and by expanding $L(1/2 + s, \chi) L(1/2 + s, \overline{\chi})$ into Dirichlet series, we obtain
\begin{align*}
& \frac{1}{2\pi i} \int_{(c)} \xi\big(\hf + w, \chi\big) \xi\big(\hf + w, \overline{\chi}\big) \frac{dw}{w} \\
&= \sum_{m,n = 1}^{\infty} \frac{\chi(m) \overline{\chi}(n)}{\sqrt{mn}} \left(\frac{q}{\pi}\right)^{1/2 + \af} \cdot \frac{1}{2\pi i} \int_{(c)} \left( \frac{q}{\pi mn}\right)^w \Gamma\left(\frac{1}{4} + \frac{w + \af}{2}\right)^2 \> \frac{dw}{w} \\
&= \left(\frac{q}{\pi}\right)^{1/2 + \af} \Gamma\left(\frac{1}{4} + \frac{\af}{2} \right)^2 \sum_{m,n = 1}^{\infty} \frac{\chi(m) \overline{\chi}(n)}{\sqrt{mn}} \Wf_{\af} \left(\frac{q}{\pi mn}\right),
\end{align*}
where 
$$ \Wf_{\af}(x) = \frac{1}{2\pi i} \int_{(c)} \frac{\Gamma\left(\frac{1}{4} + \frac{w + \af}{2}\right)^2}{\Gamma\left(\frac{1}{4} + \frac{\af}{2} \right)^2 } x^w \> \frac{dw}{w}.$$
\begin{rmk1} \label{rem:Wf}
When $x > 1,$ moving the line of integration to the left to $(-1/2 + \epsilon),$ we obtain that $\Wf_{\af}(x) = 1 + O(x^{-1/2 + \epsilon}).$ On the other hand, when $x < 1,$ we obtain that $\Wf_{\af}(x) \ll_{k} x^{k}$ for any $k > 0$ by moving the line of integration to the right. 
\end{rmk1}
Moving the the line of integration to $(-c)$ and using the functional equation
$$ \xi\big(\hf + w, \chi\big) \xi\big(\hf + w, \overline{\chi}\big) = \xi\big(\hf - w, \chi\big) \xi\big(\hf - w, \overline{\chi}\big)$$
we have
\begin{align*}
& \frac{1}{2\pi i} \int_{(c)} \xi\big(\hf + w, \chi\big) \xi\big(\hf + w, \overline{\chi}\big) \frac{dw}{w} \\
&= \left(\frac{q}{\pi}\right)^{1/2 + \af} \Gamma\left(\frac{1}{4} + \frac{\af}{2} \right)^2  |L(1/2, \chi)|^2 + \frac{1}{2\pi i} \int_{(-c)} \xi\big(\hf + w, \chi\big) \xi\big(\hf + w, \overline{\chi}\big) \frac{dw}{w} \\
&= \left(\frac{q}{\pi}\right)^{1/2 + \af} \Gamma\left(\frac{1}{4} + \frac{\af}{2} \right)^2  |L(1/2, \chi)|^2 - \frac{1}{2\pi i} \int_{(c)} \xi\big(\hf + w, \chi\big) \xi\big(\hf + w, \overline{\chi}\big) \frac{dw}{w}. 
\end{align*}
 Therefore
\begin{equation} \label{lsquareapprox}
|L(1/2, \chi)|^2 = 2 \sum_{m, n} \frac{\chi(m) \overline{\chi}(n)}{\sqrt{mn}} \Wf_{\af} \left(\frac{q}{\pi mn}\right).
\end{equation}
Since even and odd characters have different gamma factors, we will split the sum over primitive characters into the sum over even characters and odd characters. However the treatment of both cases are the same so we will concentrate on the sum over even characters. We shall use the following lemma for the orthogonality relation for the primitive, even characters.
\begin{lemma} \label{lem:sumoverevenodd} Let $\sum_{\chi}^{(e)}, \sum_{\chi}^{(o)}$ indicate the sum over non-trivial primitive even (respectively odd) characters. Then
$$ {\sum_{\chi}}^{(e)} \chi(a) = \left\{ \begin{array}{ll} \frac{\phi(q) - 2}{2} & {\rm if} \ \ a \equiv \pm 1 \ ({\rm mod} \ q ) \\
-1 & {\rm if} \ \ a  \not\equiv \pm 1 \ ({\rm mod} \ q ) \ {\rm and} \ (a,q) = 1,  \end{array} \right. $$
and 
$$ {\sum_{\chi}}^{(o)} \chi(a) = \left\{ \begin{array}{ll} \frac{\phi(q) }{2} & {\rm if} \ \ a \equiv 1 \ ({\rm mod} \ q ) \\
\frac{-\phi(q)}{2} & {\rm if} \ \ a \equiv -1 \ ({\rm mod} \ q ) \\
0 & {\rm if} \ \ a  \not\equiv \pm 1 \ ({\rm mod} \ q ) \ {\rm and} \ (a,q) = 1.  \end{array} \right.$$
\end{lemma}
\begin{proof} We can write
 $$ {\sum_{\chi}}^{(e)} \chi(a) = \sum_{\chi} \left( \frac{1 + \chi(-1)}{2} \right) \chi(a) -  1,$$
and 
$$ {\sum_{\chi}}^{(o)} \chi(a) = \sum_{\chi} \left( \frac{1 - \chi(-1)}{2} \right) \chi(a).$$
The lemma follows from the fact that $\sum_{\chi} \chi(a) = 0$ when $a \not\equiv 1$ and $\phi(q)$ otherwise. 
\end{proof}
Let us write 
$$ |\Pf|^{4s} |\Mf|^{2(2s - r)} = \sum_{a,b \leq x^{2s}y^{2s - r}} \frac{y_a y_b}{\sqrt{ab}}\chi(a)\overline{\chi}(b).$$
Note that $y_a \ll_{\epsilon} a^{\epsilon}$ for any $\epsilon.$

From (\ref{lsquareapprox}), Lemma \ref{lem:sumoverevenodd} and the fact that $x^{2s}y^{2s-r} < q^{1/20}$, we have
\begin{align*}
& {\sum_{\chi}}^{(e)} |L(1/2, \chi)|^2|\Pf|^{4s} |\Mf|^{2(2s - r)} \\
&= 2\sum_{a,b \leq x^{2s}y^{2s - r}} \frac{y_a y_b}{\sqrt{ab}} \sum_{m, n} \frac{1}{\sqrt{mn}} \Wf_{\af} \left(\frac{q}{\pi mn}\right) {\sum_{\chi}}^{(e)} \chi(ma) \overline{\chi}(nb) \\
&= \phi(q) \sum_{a,b \leq x^{2s}y^{2s - r}} \frac{y_a y_b}{\sqrt{ab}}  \sum_{\substack{m,n \\ ma \equiv \pm nb \,{\rm mod} \, q}} \frac{1}{\sqrt{mn}} \Wf_{\af} \left(\frac{q}{\pi mn}\right) + O\left( \sum_{a,b \leq x^{2s}y^{2s - r}} \frac{q^{\epsilon}}{\sqrt{ab}}  \sum_{m,n} \frac{1}{\sqrt{mn}} \Wf_{\af} \left(\frac{q}{\pi mn}\right)  \right)
\end{align*}
By Remark \ref{rem:Wf}, the error term is 
$$ \ll x^{2s}y^{2s - r}q^{\epsilon} \sum_{d} \frac{d^{\epsilon}}{\sqrt{d}} \Wf_{\af}\left(\frac{q}{\pi d}\right) \ll \frac{q}{\log q}.$$
We now estimate the contibution of the off-diagonal terms. We first consider the term $mn > q^{1 + \epsilon}.$ From Remark \ref{rem:Wf}, these terms contribute
\begin{align*}
\ll \sum_{a,b \leq x^{2s}y^{2s - r}} \frac{q^{\epsilon}}{\sqrt{ab}} \sum_{\substack{m, n \\ mn > q^{1 + \epsilon}}} \frac{1}{\sqrt{mn}} \left( \frac{q}{mn}\right)^k \ll x^{2s}y^{2s - r}q^{k + \epsilon} \sum_{d > q^{1 + \epsilon}}\frac{1}{d^{1/2 + k}} \sum_{m | d}  1 \ll \frac{q}{\log q}.
\end{align*}
Hence we assume that $mn < q^{1 + \epsilon}$ and from Remark \ref{rem:Wf}, $\Wf_{\af} \left(\frac{q}{\pi mn}\right) \ll 1.$ These terms contribute
$$ \ll q^{1 + \epsilon} \sum_{a,b \leq x^{2s}y^{2s - r}} \frac{1}{\sqrt{ab}} \sum_{\substack{m, n \\ mn \leq q^{1 + \epsilon} \\ q | ma \pm nb \\ ma \pm nb \neq 0 }} \frac{1}{\sqrt{mn}} \ll \frac{q}{\log q}$$
by the same arguments as showing (\ref{errorinlowerbound}). Now we consider the main term, $ma = nb.$ We can write $m = \frac{\alpha b}{(a,b)}$ and $n = \frac{\alpha a}{(a,b)}.$ By the definition of $\Wf_{\af}(x)$, (\ref{eqn:Pf}), and (\ref{eqn:Mf}), the main term is
\begin{align*}
& \phi(q) \sum_{a,b \leq x^{2s}y^{2s - r}} \frac{(a,b)}{ab} y_a y_b  \sum_{\alpha} \frac{1}{\alpha} \Wf_{\af} \left(\frac{q(a,b)^2}{\pi\alpha^2 ab}\right) \\
&= \frac{1}{\Gamma\left( \tfrac{1}{4}\right)^2 (\log x)^{4s}(\log y)^{8s - 4r}} \frac{\phi(q)}{(2\pi i)^{8s - 2r + 1}}\int_{(c)^{8s - 2r + 1}} \sum_{a,b} \frac{(a,b)}{ab} \psi_{\wone, \zone} (a) \psi_{\wtwo, \ztwo} (b) \Gamma\left( \frac{1}{4} + \frac{w}{2}\right)^2 \\
& \hspace{3in} \sum_{\alpha} \frac{1}{\alpha^{1 + 2w}} \left( \frac{q(a,b)^2}{\pi ab}\right)^w \frac{x^{w_1 + ... + w_{4s}}y^{z_1 + ... + z_{4s - 2r}}}{w_1^2...w_{4s}^2z_1^3...z_{4s - 2r}^3}\frac{dw}{w} \> d\w \> d\z, \\
&= \frac{1}{\Gamma\left( \tfrac{1}{4}\right)^2 (\log x)^{4s}(\log y)^{8s - 4r}} \frac{\phi(q)}{(2\pi i)^{8s - 2r + 1}}\int_{(c)^{8s - 2r + 1}} \zeta(1 + 2w) \sum_{\substack{a,b \\ (a,b) = 1}} \frac{1}{(ab)^{1 + w}} \sum_{d} \frac{\psi_{\wone, \zone} (ad) \psi_{\wtwo, \ztwo} (bd)}{d} \\
& \hspace{3in} \Gamma\left( \frac{1}{4} + \frac{w}{2}\right)^2\left( \frac{q}{\pi}\right)^w \frac{x^{w_1 + ... + w_{4s}}y^{z_1 + ... + z_{4s - 2r}}}{w_1^2...w_{4s}^2z_1^3...z_{4s - 2r}^3}\frac{dw}{w} \> d\w \> d\z,
\end{align*}
where 
$$\psi_{w_1,..., w_k, z_1,..., z_l}(a) = \sum_{a_1...a_kc_1...c_l = a} \frac{d_{1/2s}(a_1)...d_{1/2s}(a_k)d_{1/s}(c_1)...d_{1/s}(c_l)\mu(c_1)...\mu(c_l)}{a_1^{w_1}...a_k^{w_k}c_1^{z_1}...c_{l}^{z_k}}.$$
By the similar arguments to the proof of the integral of $S_l$ and letting $w = w_0$, the above can be written as $\frac{\phi(q) I_u}{ (\log x)^{4s}(\log y)^{8s-4r}}$. This proves Proposition \ref{prop:integralUpper}.

\section{The upper bound for $I_u$} \label{sec:evaluateIu}
In this section we will prove (\ref{prop:upperbound}), which is
$$ I_u \ll (\log q)^{r^2/s^2 + 12s - 4r}.$$

Recall that $I_u$ is
\begin{align*} 
&\frac{1}{\Gamma\left( \tfrac{1}{4}\right)^2} \frac{1}{(2\pi i)^{8s - 2r + 1}}\int_{(c)^{8s - 2r + 1}} \frac{\prod_{1 \leq j \leq 4s} \zeta^{1/2s}(1 + w_0 + w_j)}{\prod_{1 \leq j \leq 4s - 2r} \zeta^{1/s}(1 + w_0 + z_{j})}\\
& \hspace{1in}\cdot \frac{\prod_{1 \leq i, j \leq 2s } \zeta^{1/4s^2} (1 + w_i + w_{2s + j}) \prod_{1 \leq i, j \leq 2s-r} \zeta^{1/s^2} (1 + z_i + z_{2s - r + j}) }{\prod_{\substack{1 \leq i \leq 2s \\ 1 \leq j \leq 2s - r}} \zeta^{1/2s^2}(1 + w_i + z_{2s - r + j})\zeta^{1/2s^2}(1 + w_{2s + i} + z_{j}) }  \nonumber \\
& \hspace{1.2in} \cdot \Gamma\left( \frac{1}{4} + \frac{w_0}{2}\right)^2 \zeta(1 + 2w_0) \left( \frac{q}{\pi}\right)^{w_0} \frac{x^{w_1 + ... + w_{4s}}y^{z_1 + ... + z_{4s - 2r}}}{w_1^2...w_{4s}^2z_1^3...z_{4s - 2r}^3} \eta_u(w_0, \vec{w}, \vec{z}) \frac{dw_0}{w_0} \> d\w \> d\z \nonumber
\end{align*} 
For notational convenience, let 
\begin{align} \label{eqn:Vwz}
V(\w, \z) := &\frac{\prod_{1 \leq j \leq 4s} \zeta^{1/2s}(1 + w_j)}{\prod_{1 \leq j \leq 4s - 2r} \zeta^{1/s}(1 + z_{j})} \cdot \frac{\prod_{1 \leq i, j \leq 2s } \zeta^{1/4s^2} (1 + w_i + w_{2s + j}) \prod_{1 \leq i, j \leq 2s-r} \zeta^{1/s^2} (1 + z_i + z_{2s - r + j}) }{\prod_{\substack{1 \leq i \leq 2s \\ 1 \leq j \leq 2s - r}} \zeta^{1/2s^2}(1 + w_i + z_{2s - r + j})\zeta^{1/2s^2}(1 + w_{2s + i} + z_{j}) }   \\
&  \hspace{4in} \cdot \frac{x^{w_1 + ... + w_{4s}}y^{z_1 + ... + z_{4s - 2r}}}{w_1^2...w_{4s}^2z_1^3...z_{4s - 2r}^3}. \nonumber
\end{align}
We let $c = 1/8$ and shift the contour integral in $w_0$ to $-\tfrac{1}{8} + \frac{1}{\log q}.$ Then we have that 
$$I_u = R + R_1 + E,$$
where
\begin{align} \label{eqn:R}
R &=  \frac{1}{2} \frac{1}{(2\pi i)^{8s - 2r}} \int_{(1/8)^{8s - 2r}} \left[ \left(2\gamma  +  \frac{\Gamma '}{\Gamma} \left( \frac{1}{4} \right) + \log \frac{q}{\pi} \right)\eta_u(0, \vec{w}, \vec{z}) + \eta_u'(0, \vec{w}, \vec{z})\right]  V(\w, \z) \> d\w \> d\z, 
\end{align} 

\begin{align} \label{eqn:Rw}
R_1 &= \frac{1}{2} \frac{1}{(2\pi i)^{8s - 2r}} \int_{(1/8)^{8s - 2r}} \left( \sum_{j = 1}^{4s} \frac{\zeta '}{\zeta} (1 + w_j) + \sum_{j = 1}^{4s-2r} \frac{\zeta '}{\zeta} (1 + z_j)\right)  V(\w, \z) \  \eta_u(0, \vec{w}, \vec{z})  \> d\w \> d\z,
\end{align}
and 
\begin{align*}
E &= \frac{1}{\Gamma\left( \tfrac{1}{4}\right)^2 } \frac{1}{(2\pi i)^{8s - 2r}} \int_{(1/8)^{8s - 2r}} \int_{(-1/8 + \frac{1}{\log q})}  \Gamma\left( \frac{1}{4} + \frac{w_0}{2}\right)^2 \zeta(1 + 2w_0) \frac{\prod_{1 \leq j \leq 4s} \zeta^{1/2s}(1 + w_0 + w_j)}{\prod_{1 \leq j \leq 4s - 2r} \zeta^{1/s}(1 + w_0 + z_{j})} \\
& \hspace{2in}\cdot \frac{\prod_{1 \leq i, j \leq 2s } \zeta^{1/4s^2} (1 + w_i + w_{2s + j}) \prod_{1 \leq i, j \leq 2s-r} \zeta^{1/s^2} (1 + z_i + z_{2s - r + j}) }{\prod_{\substack{1 \leq i \leq 2s \\ 1 \leq j \leq 2s - r}} \zeta^{1/2s^2}(1 + w_i + z_{2s - r + j})\zeta^{1/2s^2}(1 + w_{2s + i} + z_{j}) }  \nonumber \\
& \hspace{3in} \left( \frac{q}{\pi}\right)^{w_0} \frac{x^{w_1 + ... + w_{4s}}y^{z_1 + ... + z_{4s - 2r}}}{w_1^2...w_{4s}^2z_1^3...z_{4s - 2r}^3} \eta_u(w_0, \vec{w}, \vec{z}) \frac{dw_0}{w_0} \> d\w \> d\z. \nonumber
\end{align*}

Here $\gamma$ is the Euler constant. Now we show that $E$ is negligible.
\begin{lemma} \label{lem:errorE}
With notation as above, $E \ll 1$. 
\end{lemma}
\begin{proof}
Indeed, since $\frac{q}{2\pi} > x^{4s}y^{4s-2r}$, we have that $ |x^{w_1 + ....+w_{4s}}y^{z_1 + ...+ z_{4s-2r}}|\left|\left(\frac{q}{2\pi}\right)^{w_0}\right| < 1 $  when  $\Re w_0 = -1/8 + \frac{1}{\log q}$, and $\Re w_i, \Re z_i = 1/8$. Let $v, v_1 = w_i$ or $z_i.$ From Theorem 6.7 in \cite{MV}, for any $\epsilon > 0,$ $ -1 < p < 1$, $\Re v$ and $\Re v_1 = 1/8$
$$ \zeta^{p}(1 + w_0 + v) \ll |w_0|^{\epsilon}|v|^{ \epsilon},$$
$$ \zeta^{p}(1 + v + v_1) \ll C.$$
From Stirling's formula, for $w_0 = -\frac{1}{8} + \frac{1}{\log q} + i \tau$
$$ \Gamma\left( \frac{1}{4} + \frac{w_0}{2}\right) \ll e^{-\pi |\tau|/2}. $$
Moreover the convexity bound gives that
$$ \zeta(1 + 2w_0) \ll (|\tau| + 1)^{1/8 + \epsilon}.$$

The integrand inside $E$ is absolutely convergent; hence it is $ \ll 1. $
\end{proof}
We will now evaluate the terms $R$ and $R_1.$  First we shift the contour at the line $\tfrac{1}{8}$ to the line $\tfrac{1}{\log x}$ for $w_1,...,w_{4s}$ and to the line $\tfrac{1}{\log y}$ for $z_1,...,z_{4s-2r}.$ We now proceed to truncate the integrals.  Let $L_1$ be a path consisting of line segments from $\tfrac{1}{\log x} + i$ to $\tfrac{1}{\log x} + i\infty$ and from $\tfrac{1}{\log x} - i\infty$ to $\tfrac{1}{\log x} - i,$ and $L_2$ is the line segment from $\tfrac{1}{\log x} - i$ to $\tfrac{1}{\log x} + i.$  $K_1$ and $K_2$ are defined similar to $L_1$ and $L_2,$ respectively, with $\log x$ replaced by $\log y.$  Note that $L_1$ and $K_1$ are bounded away from $0$ so we expect these to give a small contribution. We will start by deriving an upper bound for $R$, which reduces to bounding

\begin{align} \label{integralofeqnR} 
 \sum_{i_1,...,i_{4s}, j_1, ..., j_{4s-2r} \in \{1, 2\}}\int_{L_{i_1}}...\int_{L_{i_{4r-2s}}} \int_{K_{j_1}}...\int_{K_{j_{4s}}}  V(\w, \z) \eta(\w, \z) \> d\w \> d\z, 
\end{align} 
where $\eta(\w, \z)$ is either $\eta_u(0, \w, \z)$ or $\eta_u'(0,\w, \z).$ We expect that the bound for the integrand is $ \ll (\log q)^{\tfrac{r^2}{s^2} - 1 + 12s-4r}.$

Due to symmetry in sets $\{w_1,...,w_{2s}\}, \{w_{2s+1},..., w_{4s}\}, \{z_{1},..., z_{2s-r}\},$ and $\{z_{2s-r +1},...,z_{4s-2r}\},$ we can assume that $L_{i_1}  = ...= L_{i_{2s - t_1}} = L_{i_{2s + 1}} = ... = L_{i_{4s - t_2}} =L_2$ and $K_{j_1}  = ...= K_{j_{2s - r - v_1}} = K_{j_{2s - r + 1}} = ... = L_{j_{4s - 2r - v_2}} = K_2$, and the rest of $L_{i_k}$ is $L_1$ and $K_{j_k}$ is $K_1.$ Note that $t_1+t_2$ is the number of $w_i$ which is integrated over $L_1$ and $v_1+v_2$ is the number of $z_j$ which is integrated over $L_2$.  Hence it is enough to show that
\begin{align} \label{integralofeqnRsmall} 
\int_{(L_{2})^{4s - t_1 - t_2}} \int_{(K_2)^{4s-2r-v_1-v_2}} \int_{(L_{1})^{t_1 + t_2}} \int_{(K_{1})^{v_1 + v_2}} V(\w, \z) \eta(\w,\z) \> d\w \> d\z \ll (\log q)^{\tfrac{r^2}{s^2} - 1 + 12s-4r}. 
\end{align} 

To compute this we need the following lemma.
\begin{lemma} \label{lem:boundforzeta}
Let $r$ be a real number, $z = x$ or $y$, and $w = \frac{1}{\log z} + it$ and $v = \frac{1}{\log z} + i\tau$. Then
$$\zeta^{r} (1 + w) \ll \left\{\begin{array}{cl} \log^{|r|} (|w| + 1) \ & \  {\rm if} \  |t| \geq 1; \\ & \\
 \frac{1}{|w|^r} & \  {\rm if} \ |t| < 1,
\end{array} \right.$$
and
$$\frac{\zeta'}{\zeta} (1 + w) \ll \left\{\begin{array}{cl} \log(|w| + 1) \ & \  {\rm if} \  |t| \geq 1; \\ & \\
 \log q & \  {\rm if} \ |t| < 1.
\end{array} \right.$$
For $r > 0,$
\begin{equation} \label{lemma3:boundforpositiveR}
\zeta^{r} (1 + w + v) \ll \left\{\begin{array}{cl}  
|w+v|^{-r} \ll (\log q)^r & \ {\rm if} \ |t|, |\tau| < 1;\\
\log^{r} (\max\{|w|, |v|\} + 1 ) + (\log q)^r & \ {\rm otherwise} , 
\end{array} \right.
\end{equation}
and for $r < 0,$
\begin{equation} \label{lemma3:boundfornegativeR}
\zeta^{-r} (1 + w + v) \ll \left\{\begin{array}{cl}  
|w + v|^r & \  {\rm if} \ |t|, |\tau| < 1; \\
\log^{r} (\max\{|w|, |v|\} + 1 ) & \ {\rm otherwise}.
\end{array} \right.
\end{equation}
\end{lemma}
The proof of Lemma \ref{lem:boundforzeta} follows from Theorem 6.7 in \cite{MV}, and the fact that $x$ and $y$ are small powers of $q.$
\\

The left hand side of (\ref{integralofeqnRsmall}) is bounded above by
\begin{align} \label{eqn:beforeL1}
&\ll \int_{(L_{2})^{4s - t_1 - t_2}} \int_{(K_2)^{4s-2r-v_1-v_2}} \int_{(L_{1})^{t_1 + t_2}} \int_{(K_{1})^{v_1 + v_2}}  \frac{\prod_{1 \leq j \leq 4s} |\zeta^{1/2s}(1 + w_j)|}{\prod_{1 \leq j \leq 4s - 2r} |\zeta^{1/s}(1 + z_{j})|}\\
& \hspace{1in}\cdot \frac{\prod_{1 \leq i, j \leq 2s } |\zeta^{1/4s^2} (1 + w_i + w_{2s + j})| \prod_{1 \leq i, j \leq 2s-r} |\zeta^{1/s^2} (1 + z_i + z_{2s - r + j})| }{\prod_{\substack{1 \leq i \leq 2s \\ 1 \leq j \leq 2s - r}} |\zeta^{1/2s^2}(1 + w_i + z_{2s - r + j})\zeta^{1/2s^2}(1 + w_{2s + i} + z_{j})| }  \nonumber \\
& \hspace{2in} \cdot \frac{d|w_1|...\>d|w_{4s}| \> d|z_{1}|...\>d|z_{4s-2r}|}{|w_1|^2...|w_{4s}|^2|z_1|^3...|z_{4s - 2r}|^3}. \nonumber
\end{align}  

By (\ref{lemma3:boundforpositiveR}) in Lemma \ref{lem:boundforzeta}, if at least one of $w_i$, $w_{2s+j}$ is integrated over $L_2,$ then
\begin{equation}
\label{eq:numeratorupperbound}
 |\zeta^{1/4s^2} (1 + w_i + w_{2s + j})| \ll \log^{1/4s^2} (\max\{|w_i|, |w_{2s+j}|\} + 1 ) + (\log q)^{1/4s^2}.
\end{equation}
Otherwise,
$ |\zeta^{1/4s^2} (1 + w_i + w_{2s + j})| \ll  (\log q)^{1/4s^2}.$

The integrand above contains the term $\frac{1}{|w_1|^2....|w_{4s}|^2|z_1|^3...|z_{4s-2r}|^3}$. Therefore, we may assume $|\zeta^{1/4s^2} (1 + w_i + w_{2s + j})| \ll (\log q)^{1/4s^2}$ from (\ref{eq:numeratorupperbound}) (otherwise at least one of $w_i$ or $w_{2s+j}$ is larger than $q$ in absolute value). Similarly, we can bound $|\zeta^{1/s^2} (1 + z_i + z_{2s - r + j})| \ll (\log q)^{1/s^2}$.

If $w_i$ is on $L_1,$ and $z_{j}$ is on $K_2,$ by (\ref{lemma3:boundfornegativeR}) of Lemma \ref{lem:boundforzeta},
$ |\zeta^{-1/2s^2} (1 + w_i + z_j)| \ll \log^{1/2s^2} (|w_i| + 1 ).$
But if $z_j$ is on $K_1,$ we can conclude that
$ |\zeta^{-1/2s^2} (1 + w_i + z_j)| \ll \log^{1/2s^2} (|w_i| + 1 ) + \log^{1/2s^2} (|z_j| + 1 ) .$ Since the integral contains the term $\frac{1}{|w_i|^2}$, the integration over $w_i$ is absolutely convergent. Similarly if $z_i$ is in $K_1$, then the integration over $z_i$ is absolutely convergent. Therefore (\ref{eqn:beforeL1}) is bounded by
\begin{align*} 
&\ll (\log q)^{1 + \tfrac{(2s-r)^2}{s^2}}\int_{(K_{2})^{4s -2r - v_1-v_2}}\int_{(L_{2})^{4s - t_1 - t_2}}  \frac{\prod_{1 \leq j \leq 2s-t_1} |\zeta^{1/2s}(1 + w_j)|\prod_{1 \leq j \leq 2s-t_2} |\zeta^{1/2s}(1 + w_{j + 2s})|}{\prod_{1 \leq j \leq 2s - r-v_1} |\zeta^{1/s}(1 + z_{j})|\prod_{1 \leq j \leq 2s - r-v_2} |\zeta^{1/s}(1 + z_{2s-r+j})|}\\
& \hspace{1in}\cdot \frac{1 }{\prod_{\substack{1 \leq i \leq 2s - t_1 \\ 1 \leq j \leq 2s - r-v_2}} |\zeta^{1/2s^2}(1 + w_i + z_{2s - r + j})| \prod_{\substack{1 \leq i \leq 2s - t_2 \\ 1 \leq j \leq 2s - r-v_1}} |\zeta^{1/2s^2}(1 + w_{2s + i} + z_{j})| }  \nonumber \\
& \hspace{1.5in} \cdot \frac{d|w_1|...\>d|w_{2s - t_1}|\>d|w_{2s + 1}|\>d|w_{4s-t_2}| \>  
  d|z_{1}|...\>d|z_{2s-r-v_1}|\>d|z_{2s-r+1}|...\>d|z_{4s-2r-v_2}|}{|w_1|^2...|w_{2s-t_1}|^2|w_{2s + 1}|^2...|w_{4s - t_2}|^2|z_1|^3...|z_{2s-r-v_1}|^3|z_{4s-2r-v_2}|^3...|z_{4s - 2r-w_2}|^3} \> \nonumber \\
\end{align*}  
By Lemma \ref{lem:boundforzeta}, the above is bounded by 
\begin{align*} 
&\ll (\log q)^{1 + \tfrac{(2s-r)^2}{s^2}}\int_{(K_{2})^{4s -2r - v_1-v_2}} \int_{(L_{2})^{4s - t_1 - t_2}} \frac{\prod_{1 \leq j \leq 2s-r - v_1} |z_j|^{1/s}\prod_{1 \leq j \leq 2s-r-v_2}|z_{2s-r + j}|^{1/s}}{\prod_{1 \leq j \leq 2s -t_1} | w_{j}|^{1/2s}\prod_{1 \leq j \leq 2s -t_2} |w_{2s +j}|^{1/2s}}\\
& \hspace{1in}\cdot \prod_{\substack{1 \leq i \leq 2s - t_1 \\ 1 \leq j \leq 2s - r-v_2}} | w_i + z_{2s - r + j}|^{1/2s^2} \prod_{\substack{1 \leq i \leq 2s - t_2 \\ 1 \leq j \leq 2s - r-v_1}} | w_{2s + i} + z_{j}|^{1/2s^2}   \nonumber \\
& \hspace{1.5in} \cdot \frac{d|w_1|...\>d|w_{2s - t_1}|\>d|w_{2s + 1}|\>d|w_{4s-t_2}| \>  
  d|z_{1}|...\>d|z_{2s-r-w_1}|\>d|z_{2s-r+1}|...\>d|z_{4s-2r-v_2}| }{|w_1|^2...|w_{2s-t_1}|^2|w_{2s + 1}|^2...|w_{4s - t_2}|^2|z_1|^3...|z_{2s-r-v_1}|^3|z_{4s-2r-v_2}|^3...|z_{4s - 2r-v_2}|^3} \> \nonumber 
\end{align*}  

By the change of variable, $w_j = \frac{1 + i a_j}{\log x}$ and $z_j = \frac{1 + i b_j}{\log x}$,  and $x = y^a$, the term above is bounded by
\begin{align} \label{eqn:afterL1} 
&\ll (\log q)^{p}\int_{(-\log x, \log x)^{8s -2r - t_1 - t_2 -v_1-v_2}} 
 \prod_{\substack{1 \leq i \leq 2s - t_1 \\ 1 \leq j \leq 2s - r-v_2}} |1 + a + ia_i + ib_{2s - r + j}|^{1/2s^2}    \\
& \cdot \prod_{\substack{1 \leq i \leq 2s - t_2 \\ 1 \leq j \leq 2s - r-v_1}} |1 + a + ia_{2s + i} + ib_{j}|^{1/2s^2}  \cdot \prod_{1 \leq j \leq 2s-t_1} \frac{da_{j}}{|1 + ia_j|^{2 + 1/2s}} \prod_{1 \leq j \leq 2s-t_2} \frac{da_{j}}{|1 + ia_{2s+ j}|^{2 + 1/2s}}  \nonumber\\ 
& \cdot \prod_{1 \leq j \leq 2s-r - v_1} \frac{db_{j}}{|a + ib_j|^{3 - 1/s}} \prod_{1 \leq j \leq 2s-r - v_2} \frac{db_{2s-r +j}}{|a + ib_{2s-r+j}|^{3 - 1/s}}\nonumber \\
& \ll (\log q)^{p}\int_{(-\log q, \log q)^{8s -2r - t_1 - t_2 -v_1-v_2}} 
 \prod_{\substack{1 \leq i \leq 2s - t_1 \\ 1 \leq j \leq 2s - r-v_2}} (|1 + ia_i|^{1/2s^2} + |a +ib_{2s - r + j}|^{1/2s^2}) \nonumber\\
& \hspace{0.5in} \cdot \prod_{\substack{1 \leq i \leq 2s - t_2 \\ 1 \leq j \leq 2s - r-v_1}} (|1+ ia_{2s + i}|^{1/2s^2} + |a + ib_{j}|^{1/2s^2})  \cdot \prod_{1 \leq j \leq 2s-t_1} \frac{da_{j}}{|1 + ia_j|^{2 + 1/2s}} \prod_{1 \leq j \leq 2s-t_2} \frac{da_{j}}{|1 + ia_{2s+ j}|^{2 + 1/2s}}   
   \nonumber \\
& \hspace{0.9in} \cdot \prod_{1 \leq j \leq 2s-r - v_1}\frac{db_{j}}{|a + ib_j|^{3 - 1/s}} \prod_{1 \leq j \leq 2s-r - v_2} \frac{db_{2s-r +j}}{|a + ib_{2s-r+j}|^{3 - 1/s}},\nonumber
\end{align}  
where 
\begin{align} \label{eqnforpowerP}
p  &=  1 + \frac{(2s-r)^2}{s^2} + 4s - t_1 - t_2 + 2 (4s - 2r - v_1 - v_2) + \frac{4s-t_1-t_2}{2s} - \frac{4s-2r-v_1-v_2}{s} \\
& \ \ \ - \frac{(2s -t_1)(2s-r-v_2)}{2s^2} - \frac{(2s -t_2)(2s-r-v_1)}{2s^2}\nonumber \\
&= \frac{r^2}{s^2} - 1 + 12s - 4r - (t_1 + t_2)\left(1 - \frac{1}{s} + \frac{r}{s^2}\right) - \frac{t_1v_2}{2s^2} - \frac{t_2v_1}{2s^2} - (v_1 + v_2)\left( 2 - \frac{2}{s} \right) \nonumber \\
& \leq \frac{r^2}{s^2} - 1 + 12s - 4r.\nonumber
\end{align}
The last inequality comes from the fact that $s \geq 2$ (Remark \ref{remS}) and $t_1, t_2, v_1, v_2 \geq 0.$ Moreover, the power of each $|1 + ia_j|$ and $|a + ib_j|$ from the first two products in the integrand of (\ref{eqn:afterL1}) is at most $1/s.$ Since $s \geq 2,$ $3 - 1/s - 1/s \geq 2$ and $2 + 1/2s - 1/s \geq 5/4.$ Hence the integration in (\ref{eqn:afterL1}) is absolutely convergent and does not depend on $q.$ From (\ref{eqn:afterL1}), (\ref{eqnforpowerP}), we obtain that  $$ R \ll (\log q)^{\frac{r^2}{s^2} + 12s - 4r},$$
where the implied constant depends on $r,s$.

\begin{rmk1}\label{rmk:truncationargument}
From the method used above, we see that if at least one of $t_1,$ $t_2,$ $v_1$ or $v_2 > 0,$ then the integral contributes $o((\log q)^{r^2/s^2 + 12s -4r})$, which constitutes an error term. Therefore the main term comes from the integral over $L_2$ and $K_2.$ This method to truncate the integral to an interval of length $2$ is easily generalized to other similar integrals.  In particular, the integral for $S_l$ can be treated the same way.  In \S \ref{sec:evaluateIl}, we will truncate the integral for $S_l$ without repeating this argument.
\end{rmk1}

From Remark \ref{rmk:truncationargument}, we can truncate the integral $R_1$ to be over $L_2$ and $K_2$. Then we apply Lemma \ref{lem:boundforzeta}, which is $\frac{\zeta '}{\zeta}(1 + w_i) \ll \log q,$ and $\frac{\zeta '}{\zeta}(1 + z_i) \ll \log q,$ when $w_i$ is on $L_2$ and $z_j$ is on $K_2.$ By the same arguments as bounding $R,$ we can show that
$$ R_1 \ll (\log q)^{\frac{r^2}{s^2} + 12s - 4r}.$$
Combining the bound on $R$ and $R_1,$ we derive (\ref{prop:upperbound}).

\section{The lower bound for $I_l$} \label{sec:evaluateIl}
In this section we will prove (\ref{prop:lowerbound}), which is 
$$ I_l \gg (\log x)^{2s + 1} (\log y)^{r^2/s^2 + 4s-4r - 1}.$$
Recall that $I_l$ is
\begin{align*} 
&  \frac{1}{(2\pi i)^{4s-2r + 1} }\int_{(c)^{4s - 2r + 1}} \frac{\prod_{1 \leq i \leq 2s} \zeta^{1/2s}(1 + w_0 + w_{i}) \prod_{1 \leq  j, k \leq s - r} \zeta^{1/s^2}(1 + z_j + z_{s - r + k})}{ \prod_{\substack{1 \leq i \leq 2s \\ 1 \leq j \leq s - r}} \zeta^{1/2s^2}(1 + w_{i} + z_{j}) \prod_{1 \leq k \leq s - r} \zeta^{1/s}(1 + w_0 + z_{s - r + k})}  \\
& \ \ \ \ \ \ \ \ \ \ \ \ \ \ \ \ \ \ \ \ \ \ \ \ \ \ \ \ \ \ \ \ \ \ \ \ \ \ \ \ \ \ \ \ \ \ \ \ \ \ \ \ \ \ \ \ \ \ \ \ \cdot X^{w_0}\Gamma(w_0) \frac{x^{w_1 + ...+ w_{2s}}y^{ z_1 + ... + z_{2(s - r)}}}{w_1^2...w_{2s}^2z_1^3...z_{2(s-r)}^3}  \eta_l(w_0, \vec{w}, \vec{z}) \> dw_0 \> d\vec{w} \> d\vec{z}. \nonumber
\end{align*}

Let $c = \frac{1}{8}$ in the integral above. Then we shift $w_0$ to $-\frac{1}{8} + \frac{1}{\log q}$ and obtain
$ I_l = M + E_1,$
where
\begin{align*} 
&M = \frac{1}{(\log x)^{2s}(\log y)^{4s-4r}} \frac{1}{(2\pi i)^{4s-2r}}\int_{(1/8)^{4s - 2r }} \frac{\prod_{1 \leq i \leq 2s} \zeta^{1/2s}(1 + w_{i}) \prod_{1 \leq  j, k \leq s - r} \zeta^{1/s^2}(1 + z_j + z_{s - r + k})}{ \prod_{\substack{1 \leq i \leq 2s \\ 1 \leq j \leq s - r}} \zeta^{1/2s^2}(1 + w_{i} + z_{j}) \prod_{1 \leq k \leq s - r} \zeta^{1/s}(1 + z_{s - r + k})}  \\
& \ \ \ \ \ \ \ \ \ \ \ \ \ \ \ \ \ \ \ \ \ \ \ \ \ \ \ \ \ \ \ \ \ \ \ \ \ \ \ \ \ \ \ \ \ \ \ \ \ \ \ \ \ \ \ \ \ \ \ \ \cdot  \frac{x^{w_1 + ... + w_{2s}}y^{ z_1 + ... + z_{2(s - r)}}}{w_1^2...w_{2s}^2z_1^3...z_{2(s-r)}^3}  \eta_l(0, \vec{w}, \vec{z})  \> d\vec{w} \> d\vec{z}, \nonumber
\end{align*}
and 
\begin{align*} 
&E_1 := \frac{1}{(\log x)^{2s}(\log y)^{4s-4r}} \frac{1}{(2\pi i)^{4s-2r + 1} }\int_{(1/8)^{4s - 2r }} \int_{(-1/8 + 1/\log q)}\\
& \hspace{2in} \frac{\prod_{1 \leq i \leq 2s} \zeta^{1/2s}(1 + w_0 + w_{i}) \prod_{1 \leq  j, k \leq s - r} \zeta^{1/s^2}(1 + z_j + z_{s - r + k})}{ \prod_{\substack{1 \leq i \leq 2s \\ 1 \leq j \leq s - r}} \zeta^{1/2s^2}(1 + w_{i} + z_{j}) \prod_{1 \leq k \leq s - r} \zeta^{1/s}(1 + w_0 + z_{s - r + k})}  \\
& \ \ \ \ \ \ \ \ \ \ \ \ \ \ \ \ \ \ \ \ \ \ \ \ \ \ \ \ \ \ \ \ \ \ \ \ \ \ \ \ \ \ \ \ \ \ \ \ \ \ \ \ \ \ \ \ \ \ \ \ \cdot X^{w_0}\Gamma(w_0) \frac{x^{w_1 + ... w_{2s}}y^{ z_1 + ... + z_{2(s - r)}}}{w_1^2...w_{2s}^2z_1^3...z_{2(s-r)}^3}  \eta_l(w_0, \vec{w}, \vec{z}) \> dw_0 \> d\vec{w} \> d\vec{z}. \nonumber
\end{align*}

By the same arguments as the proof of Lemma \ref{lem:errorE}, we can show that $E_1 \ll 1.$ Now, we will focus on bounding $M.$ We first shift the contour of integration in $w_i$ to $\Re w_i = 1/\log x$ for all $i$ and similarly shift the contour of integration in $z_j$ to $\Re z_j = 1/\log y$ for all $j.$  

We will now truncate the integrals.  We will want to take Taylor expansions so we will truncate further than in the last section.  Let $L_1$ be the path consisting of line segments from $\tfrac{1}{\log x} + i$ to $\tfrac{1}{\log x} + i\infty$ and from $\tfrac{1}{\log x} - i\infty$ to $\tfrac{1}{\log x} - i,$ $L_2$ the path consisting of line segments from $\tfrac{1}{\log x} + \frac{i}{\sqrt{\log x}}$ to $\tfrac{1}{\log x} + i$ and from $\tfrac{1}{\log x} - i$ to $\tfrac{1}{\log x} - \frac{i}{\sqrt{\log x}},$ and $L_3$ the line segment from $\tfrac{1}{\log x} - \frac{i}{\sqrt{\log x}}$ to $\tfrac{1}{\log x} + \frac{i}{\sqrt{\log x}}.$ Define $K_1, K_2$ and $K_3$ similarly but replacing every appearance of $\log x$ with $\log y.$ Hence $M$ can be written as

\begin{align*} 
&M = \frac{1}{(\log x)^{2s}(\log y)^{4s-4r}} \frac{1}{(2\pi i)^{4s-2r} }\sum_{i_1,...,i_{2s}, j_1,...,j_{2s-2r} \in \{1, 2, 3\}}\int_{L_{i_1}}...\int_{L_{i_{2s}}} \int_{K_{j_1}}...\int_{K_{j_{2s-2r}}} \\
& \ \ \ \ \ \ \ \frac{\prod_{1 \leq i \leq 2s} \zeta^{1/2s}(1 + w_{i}) \prod_{1 \leq  j, k \leq s - r} \zeta^{1/s^2}(1 + z_j + z_{s - r + k})}{ \prod_{\substack{1 \leq i \leq 2s \\ 1 \leq j \leq s - r}} \zeta^{1/2s^2}(1 + w_{i} + z_{j}) \prod_{1 \leq k \leq s - r} \zeta^{1/s}(1 + z_{s - r + k})}  \cdot  \frac{x^{w_1 + ... w_{2s}}y^{ z_1 + ... + z_{2(s - r)}}}{w_1^2...w_{2s}^2z_1^3...z_{2(s-r)}^3}  \eta_l(0, \vec{w}, \vec{z})  \> d\vec{w} \> d\vec{z}. \nonumber
\end{align*}

From Remark \ref{rmk:truncationargument}, if at least one of the integrals is over $L_1$ or $K_1,$ then the contribution of the integrals is small. Hence we can assume that $i_1,..,i_{2s}, j_1,...,j_{2s - 2r} \in \{2, 3\}.$

Since there are symmetries among variables $\{w_1,..., w_{2s}\},$ $\{z_1,...,z_{s-r}\},$ $\{z_{s-r + 1} ,...,z_{2s-2r}\},$ we can assume that 
\begin{enumerate}
\item $L_{i_1} = ...= L_{i_{2s - t}} = L_3,$
\item $L_{i_{2s - t + 1}} = ... = L_{i_{2s}} = L_2,$ 
\item $K_{j_1} = .... = K_{j_{s-r - v_1}} = K_{j_{s-r + 1}} = ... = K_{j_{2s-2r - v_2}} = K_3,$
\item $K_{j_{s-r-v_1 + 1}} = ... = K_{j_{s-r}} = K_{j_{2s-2r - v_2 + 1}} ... = K_{j_{2s-2r}} = K_2.$ 
\end{enumerate}

Note that $t$ is the number of $w_i$ which is integrated over $L_2$ and $v_1+v_2$ is the number of $z_j$ integrated over $K_2$.  

\begin{lemma}
With notation as above, if at least one of $t, v_1, v_2$ is not zero, then 
 \begin{align}  \label{integral:errorboundLower}
& \int_{K_{3}^{2s - 2r - v_1 - v_2}}\int_{K_{2}^{v_1 + v_2}} \int_{L_{3}^{2s - t}}\int_{L_{2}^{t}} \frac{\prod_{1 \leq i \leq 2s} \zeta^{1/2s}(1 + w_{i}) \prod_{1 \leq  j, k \leq s - r} \zeta^{1/s^2}(1 + z_j + z_{s - r + k})}{ \prod_{\substack{1 \leq i \leq 2s \\ 1 \leq j \leq s - r}} \zeta^{1/2s^2}(1 + w_{i} + z_{j}) \prod_{1 \leq k \leq s - r} \zeta^{1/s}(1 + z_{s - r + k})}\\
& \hspace{3in}   \cdot  \frac{x^{w_1 + ...+ w_{2s}}y^{ z_1 + ... + z_{2(s - r)}}}{w_1^2...w_{2s}^2z_1^3...z_{2(s-r)}^3}  \eta_l(0, \vec{w}, \vec{z})  \> d\vec{w} \> d\vec{z} \nonumber \\
& = o \left( (\log q)^{\frac{r^2}{s^2} + 6s - 4r}\right). \nonumber
\end{align}
\end{lemma}
This will imply that the main contribution of $M$ will come from the integral over $L_3$ for $w_i$ and $K_3$  for $z_j$ for all $1 \leq i \leq 2s$ and $1 \leq j \leq 2s - 2r.$ 

\begin{proof}
Since we assume that $z_j, z_{s-r+k}$ is on either $K_2$ or $K_3,$ from (\ref{lemma3:boundforpositiveR}) in Lemma \ref{lem:boundforzeta}, we obtain that $$|\zeta^{1/s^2}(1 + z_j + z_{s-r + k})| \ll \log^{1/s^2} q.$$ 
Also $ x^{w_1 + ... w_{2s}}y^{ z_1 + ... + z_{2(s - r)}} \ll 1$ on $L_i$ and $K_i$.

When $w_i$ is on $L_2$ or $L_3,$ and $z_j$ is on $K_2$ or $K_3,$  by Lemma \ref{lem:boundforzeta}, 
$$|\zeta^{1/2s}(1 + w_i)| \ll |w_i|^{-1/2s};$$
$$ |\zeta^{-1/s}(1 + z_{s - r + k}) | \ll |z_{s-r+k}|^{1/s};$$
and 
$$ |\zeta^{-1/2s^2}(1 + w_i + z_j)| \ll |w_i + z_j|^{1/2s^2}.$$

Therefore (\ref{integral:errorboundLower}) is bounded above by
\begin{align}  \label{integral:errorboundLower2}
&(\log q)^{\tfrac{(s-r)^2}{s^2}} \int_{K_{3}^{2s - 2r - v_1 - v_2 }}\int_{K_{2}^{v_1 +v_2}} \int_{L_{3}^{2s - t}}\int_{L_{2}^{t}}\prod_{\substack{1 \leq i \leq 2s  \\ 1 \leq j \leq s - r }} |w_{i} + z_{j} |^{1/2s^2}\\
& \hspace{2 in}\prod_{1 \leq i \leq 2s } \frac{d|w_i|}{|w_i|^{2+ 1/2s}} \prod_{1 \leq j \leq s-r}\frac{d|z_{j}|}{|z_j|^{3}} \prod_{1 \leq k \leq s-r}  \frac{d|z_{s-r + k}|}{|z_{s-r+k}|^{3 - 1/s}}. \nonumber 
\end{align}
Note that 
$ |w_{i} + z_{j} |^{1/2s^2} \leq |w_{i}|^{1/2s^2} + |z_{j} |^{1/2s^2}.$
Hence 
$\prod_{\substack{1 \leq i \leq 2s \\ 1 \leq j \leq s - r }} |w_{i} + z_{j} |^{1/2s^2}$ is bounded by $$ \sum_{\substack{0 \  \leq \ \alpha_1, ..., \alpha_{2s } \leq \frac{s-r}{2s^2} \\ 0 \ \leq \ \beta_1, ...., \beta_{s-r}  \leq \frac{2s}{2s^2}} }
|w_1|^{\alpha_{1}}...|w_{2s}|^{\alpha_{2s}}|z_1|^{\beta_1}...|z_{s-r}|^{\beta_{s-r}},$$
where 
$$\alpha_1 + ...+ \alpha_{2s-t} + \beta_1 + ... + \beta_{s-r} = \frac{s-r}{s}.$$
Each term of the sum in the integral (\ref{integral:errorboundLower2}) is bounded by
\begin{align}  \label{integral:errorboundLower3}
&(\log q)^{\tfrac{(s-r)^2}{s^2}} \int_{K_{3}^{2s - 2r - v_1 - v_2 }}\int_{K_{2}^{v_1 + v_2}} \int_{L_{3}^{2s - t}}\int_{L_{2}^{t}}\prod_{1 \leq i \leq 2s } \frac{d|w_i|}{|w_i|^{2+ 1/2s - \alpha_i}} \prod_{1 \leq j \leq s-r}\frac{d|z_{j}|}{|z_j|^{3 - \beta_j}} \prod_{1 \leq k \leq s-r}  \frac{d|z_{s-r + k}|}{|z_{s-r+k}|^{3 - 1/s}}. 
\end{align}

Since $\alpha_i \leq \frac{s-r}{2s^2},$  we have $2 + 1/2s - \alpha_i \geq 2 + \frac{r}{2s^2} > 2.$ Moreover, $\beta_i \leq \frac{1}{s}$, and $s \geq 2$. Thus $3 - \beta_i \geq  5/2$ and $3 - 1/s \geq 5/2.$

Integrating $w_i$ over $L_2$, we obtain that 
$$ \int_{L_2} \frac{d|w_i|}{|w_i|^{2 + 1/2s - \alpha_i}} \ll\int_{\frac{1}{\sqrt{\log x}}}^{1} \frac{1}{t^{2 + 1/2s - \alpha_i}} \>dt \ll (\log q)^{1/2 + 1/4s - \alpha_i/2}.$$

Integrating $w_i$ over $L_3,$ we obtain that
\begin{align*}
 \int_{L_3} \frac{d|w_i|}{|w_i|^{2 + 1/2s - \alpha_i}} &\ll (\log q)^{1 + 1/2s - \alpha_i} \int^{\sqrt{\log x}}_{-\sqrt{\log x}} \frac{1}{|1 + it|^{2 + 1/2s - \alpha_i}} \>dt \\
&\ll (\log q)^{1 + 1/2s - \alpha_i} \int^{\infty}_{-\infty} \frac{1}{1 + t^2} \>dt \\
&\ll (\log q)^{1 + 1/2s - \alpha_i}.
\end{align*}
Similary integrating $z_j$ and $z_{s-r + k}$ over $K_2$, we have
$$ \int_{K_2} \frac{d|z_j|}{|z_j|^{3 - \beta_j}} \ll (\log q)^{1 - \beta_j/2},
\ \ \ \ \ \ {\rm and } \ \ \ \ \ \ \int_{K_2} \frac{d|z_{s-r+k}|}{|z_{s-r+k}|^{3 - 1/s}} \ll (\log q)^{1 - 1/2s}.$$
Integrating $z_j$ and  $z_{s-r + k}$ over $K_3$, we have
$$ \int_{K_3} \frac{d|z_j|}{|z_j|^{3 - \beta_j}} \ll (\log q)^{2 - \beta_j},
\ \ \ \ \ \ {\rm and } \ \ \ \ \ \ \int_{K_3} \frac{d|z_{s-r+k}|}{|z_{s-r+k}|^{3 - 1/s}} \ll (\log q)^{2 - 1/s}. $$
Therefore (\ref{integral:errorboundLower3}) is bounded above by
$ (\log q)^{p},$
where 
\begin{align*}
 p = &\frac{(s-r)^2}{s^2}   + \sum_{i = 2s - t + 1}^{2s } \left(\frac{1}{2} + \frac{1}{4s} - \frac{\alpha_i}{2}\right) + \sum_{i = 1}^{2s - t} \left( 1 + \frac{1}{2s} - \alpha_i \right)+ \sum_{j = s-r-v_1 + 1}^{j = s-r }\left( 1 - \frac{\beta_j}{2}\right)  + \sum_{j = 1}^{s-r-v_1}(2 - \beta_j) \\
&+ v_2 \left( 1 - \frac{1}{2s}\right) + (s - r - v_2)\left(2 - \frac{1}{s}\right).
\end{align*}
Since $0 \leq \alpha_i \leq \frac{s - r}{2s^2} < 1,$ $0 \leq \beta_j \leq \frac{1}{s} < 1,$ and $s \geq 2,$ we have 
$$\frac{1}{2} + \frac{1}{4s} - \frac{\alpha_i}{2} <  1 + \frac{1}{2s} - \alpha_i \ \ \ \ \ \ ; \ \ \  1 - \frac{\beta_j}{2} < 2 - \beta_j \ \ \ \ \ {\rm and } \\ \ \ \ 1 - \frac{1}{2s} < 2 - \frac{1}{s}.$$

If at least one of $t, v_1, v_2$ is greater than 0, then 
\begin{align*} 
 p \ \ &< \ \ \frac{(s-r)^2}{s^2}  + \sum_{i = 1}^{2s } \left( 1 + \frac{1}{2s} - \alpha_i \right) +  \sum_{j = 1}^{s-r}(2 - \beta_j) + (s - r )\left(2 - \frac{1}{s}\right)  \\
&=  6s - 4r + \frac{r^2}{s^2} . \nonumber
\end{align*}
We see that if at least one of $t, v_1, v_2$ is greater than 0, then (\ref{integral:errorboundLower3}) is $o((\log q)^{6s - 4r + \frac{r^2}{s^2}})$, where the implied constant depends only on $r$ and $s,$ and we arrive at (\ref{integral:errorboundLower}). 
\end{proof}

We will now bound the main contribution, which is
\begin{align} \label{integral:maincontributionL3K3}
&\frac{1}{(2\pi i)^{4s-2r} }\int_{K_{3}^{2s-2r}}\int_{L_{3}^{2s}} \frac{\prod_{1 \leq i \leq 2s} \zeta^{1/2s}(1 + w_{i}) \prod_{1 \leq  j, k \leq s - r} \zeta^{1/s^2}(1 + z_j + z_{s - r + k})}{ \prod_{\substack{1 \leq i \leq 2s \\ 1 \leq j \leq s - r}} \zeta^{1/2s^2}(1 + w_{i} + z_{j}) \prod_{1 \leq k \leq s - r} \zeta^{1/s}(1 + z_{s - r + k})} \\
& \hspace{3 in} \cdot  \frac{x^{w_1 + ...+ w_{2s}}y^{ z_1 + ... + z_{2(s - r)}}}{w_1^2...w_{2s}^2z_1^3...z_{2(s-r)}^3}  \eta_l(0, \vec{w}, \vec{z})  \> d\vec{w} \> d\vec{z}. \nonumber
\end{align} 
To obtain a lower bound for (\ref{integral:maincontributionL3K3}), we need the following two lemmas.
\begin{lemma} \label{lem:uppforL3K3} Let $s, r$ be positive integers defined as in \S 2 and $t$ be an integer such that $0 \leq t \leq 2s$.  Further, let $\alpha_i$ be a rational number greater than 2 such that $\alpha_i - \frac{t}{2s^2} - \frac{1}{s} \geq 2.$ Then 
\begin{align*} 
&\int_{K_{3}^{2s-2r}} \int_{L_{3}^{t}}\frac{\prod_{1 \leq i \leq t} \zeta^{1/2s}(1 + w_{i}) \prod_{1 \leq  j, k \leq s - r} \zeta^{1/s^2}(1 + z_j + z_{s - r + k})}{ \prod_{\substack{1 \leq i \leq t \\ 1 \leq j \leq s - r}} \zeta^{1/2s^2}(1 + w_{i} + z_{j}) \prod_{1 \leq k \leq s - r} \zeta^{1/s}(1 + z_{s - r + k})} \\
& \hspace{3 in} \cdot  \frac{x^{w_1 + ...+ w_{t}}y^{ z_1 + ... + z_{2(s - r)}}}{w_1^2...w_{t}^2z_1^{\alpha_1}...z_{2(s-r)}^{\alpha_{2(s-r)}}}  \eta(\vec{w}, \vec{z})  \> d\vec{w} \> d\vec{z}, \nonumber \\
&\ll(\log x)^{\tfrac{t}{2s} + t}(\log y)^{\tfrac{(s-r)^2}{s^2} -\tfrac{t(s-r)}{2s^2} - \tfrac{s-r}{s} + \alpha_1 + ... + \alpha_{2s-2r} - (2s-2r)}
\end{align*} 
where $\eta(\vec{w}, \vec{z})$ is an absolutely convergent Euler product for $w_i$, $z_i$ in the domain ${\rm Re} (w_i), {\rm Re} (z_i) \geq -3/16,$ and the implied constant depends on $s, r.$ 
\end{lemma}
\begin{proof} Since $w_i$ is on $L_3$ and $z_i$ is on $K_3,$ by Lemma \ref{lem:boundforzeta}, the integral in the statement is bounded by 
\begin{align*} 
&\int_{K_{3}^{2s-2r}} \int_{L_{3}^{t}}\frac{\prod_{\substack{1 \leq i \leq t \\ 1 \leq j \leq s - r}} | w_{i} + z_{j}|^{1/2s^2} }{\prod_{1 \leq  j, k \leq s - r} | z_j + z_{s - r + k}|^{1/s^2}} \cdot  \prod_{1 \leq i \leq t} \frac{d|w_i|}{|w_i|^{2 + 1/2s}}\prod_{1 \leq j \leq s - r}\frac{d|z_j|}{|z_j|^{\alpha_j}}  \prod_{1 \leq k \leq s-r} \frac{d|z_{s-r + k}|}{|z_{s - r + k}|^{\alpha_{s-r+k} - 1/s}}.\nonumber \\
\end{align*}

Changing variables from $w_i$ to $\frac{1 + it_i}{\log x}$ and $z_j$ to $\frac{1 +i\tau_j}{\log y},$ where $-\sqrt{\log x} \leq t_i \leq \sqrt{\log x}$ and $-\sqrt{\log y} \leq \tau_i \leq \sqrt{\log y},$ we obtain that the above expression is bounded above by 
\begin{equation} \label{eqn:lemma4logJ}
 (\log x)^{t + t/2s} (\log y)^{\tfrac{(s-r)^2}{s^2} + \alpha_1 + ... + \alpha_{2s - 2r} - (2s - 2r) -\frac{s-r}{s}} \cdot J,
\end{equation}
where
\begin{align*} 
& J = \int_{(-\sqrt{\log y}, \sqrt{\log y})^{2s-2r}} \int_{(-\sqrt{\log x}, \sqrt{\log x})^{t}}\frac{\prod_{\substack{1 \leq i \leq t \\ 1 \leq j \leq s - r}} \left| \frac{1 + it_i}{\log x} + \frac{1 + i\tau_j}{\log y}\right|^{1/2s^2} }{\prod_{1 \leq  j, k \leq s - r} | 2 + i\tau_{j} + i\tau_{s - r + k}|^{1/s^2}} \\
& \hspace{1 in} \cdot  \prod_{1 \leq i \leq t} \frac{dt_i}{|1 + it_i|^{2 + 1/2s}}\prod_{1 \leq j \leq s - r}\frac{d\tau_j}{|1 + i\tau_j|^{\alpha_j}}  \prod_{1 \leq k \leq s-r} \frac{d\tau_{s-r + k}}{|1 + i\tau_{s - r + k}|^{\alpha_{s-r+k} - 1/s}}.\nonumber \\
& \leq \int_{(-\infty, \infty)^{2s-2r}} \int_{(-\infty, \infty)^{t}}\prod_{\substack{1 \leq i \leq t \\ 1 \leq j \leq s - r}} \left(\left| \frac{1 + it_i}{\log x} \right|^{1/2s^2} + \left|\frac{1 + i\tau_j}{\log y}\right|^{1/2s^2} \right) \\
& \hspace{1 in} \cdot  \prod_{1 \leq i \leq t} \frac{dt_i}{|1 + it_i|^{2 + 1/2s}}\prod_{1 \leq j \leq s - r}\frac{d\tau_j}{|1 + i\tau_j|^{\alpha_j}}  \prod_{1 \leq k \leq s-r} \frac{d\tau_{s-r + k}}{|1 + i\tau_{s - r + k}|^{\alpha_{s-r+k} - 1/s}}.\nonumber \\
& \leq (\log y)^{-\tfrac{t(s-r)}{2s^2}}\int_{(-\infty, \infty)^{2s-2r}} \int_{(-\infty, \infty)^{t}}\prod_{\substack{1 \leq i \leq t \\ 1 \leq j \leq s - r}} \left(\left| 1 + it_i\right|^{1/2s^2} + \left|1 + i\tau_j\right|^{1/2s^2} \right) \\
& \hspace{1 in} \cdot  \prod_{1 \leq i \leq t} \frac{dt_i}{|1 + it_i|^{2 + 1/2s}}\prod_{1 \leq j \leq s - r}\frac{d\tau_j}{|1 + i\tau_j|^{\alpha_j}}  \prod_{1 \leq k \leq s-r} \frac{d\tau_{s-r + k}}{|1 + i\tau_{s - r + k}|^{\alpha_{s-r+k} - 1/s}},\nonumber
\end{align*}
where the last inequality comes from the fact that $\log x > \log y.$ Note that $\alpha_j - \frac{t}{2s^2} - 1/s \geq 2$ and moreover $2 + 1/2s - \frac{s - r}{2s^2} \geq 2 + \frac{r}{2s^2}.$ Therefore the integrations over $t_i$ and $\tau_j$ are absolutely convergent, and $J \ll (\log y)^{-\frac{t(s-r)}{2s^2}}.$ Inserting this bound into (\ref{eqn:lemma4logJ}) completes the proof. 
\end{proof}

\begin{lemma} \label{lem:lowerboundz}
Fix any real $\alpha > 2$ and $\beta > 0$ and any natural number $m$.  Also, let $y > 2$ and write
$$I = \left(\frac{1}{2\pi i}\right)^{2m}  \int_{(\frac{1}{\log y})^{2m}} \prod_{1\leq i, j \leq m} \zeta^\beta (1+z_i+z_{m+j}) \frac{y^{\sum_{1\leq i\leq 2m}z_i}}{\prod_{1\leq i\leq 2m}z_i^\alpha} d\vec z,
$$where $\int_{(\frac{1}{\log y})^{2m}}$ denotes an iterated integral on the vertical lines $\Re z_i = \frac{1}{\log y}$ for $1\leq i\leq 2m$.  Let $\gamma = 2m\alpha + m^2 \beta - 2m$.  Then
$$I \gg_m (\log y)^{\gamma}.
$$
\end{lemma}
\begin{proof}
Since we are in the region of absolute convergence, write
$$
\zeta^\beta (1+z_i+z_{m+j})= \sum_{n_{i, j}\geq 1} \frac{d_\beta(n_{i, j})}{n_{i, j}} \frac{1}{n^{z_i}n^{z_{m+j}}},
$$so that

\begin{eqnarray*}
I &=& \prod_{1\leq i, j \leq m}\sum_{n_{i, j}=1}^y \frac{d_\beta(n_{i, j})}{n_{i, j}}  \left(\frac{1}{2\pi i}\right)^{2m}  \int_{(\frac{1}{\log y})^{2m}} \frac{\prod_{1\leq i\leq m} \left(\frac{y}{\prod_{1\leq j\leq m}n_{i, j}}\right)^{z_i}\prod_{1\leq j\leq m} \left(\frac{y}{\prod_{1\leq i\leq m}n_{i, j}}\right)^{z_j}} {\prod_{1\leq i\leq 2m}z_i^\alpha} d\vec z\\
&=& \prod_{1\leq i, j \leq m} {\sum_{n_{i, j}=1}^y}^{\flat} \frac{d_\beta(n_{i, j})}{n_{i, j}}  \prod_{1\leq i\leq m} \left( \frac{1}{\Gamma(\alpha)} \log^{\alpha - 1} \left(\frac{y}{\prod_{1\leq j\leq m} n_{i, j} } \right) + O\left((\log y)^{\alpha - 2}\right) \right)\\ 
&& \cdot \prod_{1\leq j\leq m} \left(\frac{1}{\Gamma(\alpha)}\log^{\alpha - 1} \left(\frac{y}{\prod_{1\leq i\leq m}n_{i, j}}\right)+O\left((\log y)^{\alpha - 2} \right)\right),
\end{eqnarray*}
where ${\sum_{n_{i, j}=1}^y}^{\flat}$ denotes sums over $n_{i, j}$ such that $\prod_{1\leq i\leq m}n_{i, j} < y$ for all $j$ and $\prod_{1\leq j\leq m}n_{i, j}<y$ for all $i$.  Since the coefficient $d_\beta(n)$ is positive, we may truncate each sum at $n_{i, j} \leq y^{\frac{1}{2m}}$ and still obtain a lower bound.  With this truncation, we have $\log^{\alpha - 1} \frac{y}{\prod_{1\leq j\leq m} n_{i, j} } + O\left((\log y)^{\alpha - 2}\right) \asymp (\log y)^{\alpha - 1}$.  Then
\begin{eqnarray*}
I &\gg& (\log y)^{2m(\alpha-1)} \prod_{1\leq i, j \leq m} \sum_{n_{i, j}=1}^{y^{\frac{1}{2m}}} \frac{d_\beta(n_{i, j}) }{n_{i, j}}\\
&\asymp& (\log y)^{2m(\alpha-1)} \left(\log^{\beta} y^{\frac{1}{2m}}\right)^{m^2}\\
&\asymp& (\log y)^{\gamma},
\end{eqnarray*}as desired.

\end{proof}

Now we will start bounding $(\ref{integral:maincontributionL3K3}).$ On $L_3, K_3$, we can write $\eta_l(0, \vec{w}, \vec{z})$ as
$$\eta_l(0,0,0)\left(1 + O(|w_1| + ... + |w_4| + |z_1| + |z_2|)\right) =\eta_l(0,0,0)\left(1 + O\left(\frac{1}{\sqrt{\log q}}\right)\right).$$
From Lemma \ref{lem:uppforL3K3}, the integral (\ref{integral:maincontributionL3K3}) is $O((\log q)^{r^2/s^2 + 6s-2r}).$
Therefore the integral of the big-$O$ term in the Taylor expansion is $o\left((\log q)^{r^2/s^2 + 6s-2r}\right),$ and we will focus on bounding 
\begin{align}  \label{eqn:M1}
M_1 &=\frac{1}{(2\pi i)^{4s-2r} }\int_{K_{3}^{2s-2r}} \int_{L_{3}^{2s}}\frac{\prod_{1 \leq i \leq 2s} \zeta^{1/2s}(1 + w_{i}) \prod_{1 \leq  j, k \leq s - r} \zeta^{1/s^2}(1 + z_j + z_{s - r + k})}{ \prod_{\substack{1 \leq i \leq 2s \\ 1 \leq j \leq s - r}} \zeta^{1/2s^2}(1 + w_{i} + z_{j}) \prod_{1 \leq k \leq s - r} \zeta^{1/s}(1 + z_{s - r + k})} \\
& \hspace{3 in} \cdot  \frac{x^{w_1 + ...+ w_{2s}}y^{ z_1 + ... + z_{2(s - r)}}}{w_1^2...w_{2s}^2z_1^3...z_{2(s-r)}^3}   \> d\vec{w} \> d\vec{z}. \nonumber
\end{align} 

The calculuation in deriving the lower bound for $M_1$ is delicate and the notation becomes rather involved. For clarity, we will first illustrate the method by working out the case $k = 1/2$ (i.e. $r = 1$ and $s = 2$), and then indicate the changes necessary for the general case for $k$ any rational number between 0 and 1.

\subsection{k = 1/2}\label{sec:1/2case}
In this section, we may write bounds such as $G = O(f(x, y)/\sqrt{a})$ or $G \ll f(x, y)/\sqrt{a}$ for certain functions $f$. Here, it is important that the implied constant does not depend on $a$.    For the case of $k = 1/2$, $M_1$ is 
\begin{align} \label{M1for12}
& \frac{1}{(2\pi i)^{6} }\int_{K_{3}^{2}} \int_{L_{3}^{4}}\frac{\prod_{1 \leq i \leq 4} \zeta^{1/4}(1 + w_{i}) \zeta^{1/4}(1 + z_1 + z_{2})}{ \prod_{1 \leq i \leq 4} \zeta^{1/8}(1 + w_{i} + z_{1}) \zeta^{1/2}(1 + z_{2})}  \cdot  \frac{x^{w_1 + ... +w_{4}}y^{ z_1 +z_{2}}}{w_1^2...w_{4}^2z_1^3z_{2}^3}   \> d\vec{w} \> d\vec{z}
\end{align}  

Recall that $x = y^a.$ We will first shift $w_4$ and replace $L_3$ by a path consisting of $\mathcal{C}_1, \mathcal{C}_2,$ and $\mathcal{C}_3.$ $\mathcal{C}_1$ composes of two horizontal line segments with vertices from $-\frac{\sqrt{a}}{\log x} + \frac{i}{\sqrt{\log x}}$ to $\frac{1}{\log x} + \frac{i}{\sqrt{\log x}}$ and from $\frac{1}{\log x} -  \frac{i}{\sqrt{\log x}}$ to $-\frac{\sqrt{a}}{\log x} - \frac{i}{\sqrt{\log x}}.$ $\mathcal{C}_2$ consists of two vertical line segments with vertices from $-\frac{\sqrt{a}}{\log x} + \frac{i}{\log x}$ to $-\frac{\sqrt{a}}{\log x} + \frac{i}{\sqrt{\log x}}$ and  from $-\frac{\sqrt{a}}{\log x} - \frac{i}{\sqrt{\log x}}$ to $-\frac{\sqrt{a}}{\log x} - \frac{i}{\log x}.$ Finally $\mathcal{C}_3$ is a polygonal beginning with the line segment from $-\frac{\sqrt{a}}{\log x} - \frac{i}{\log x}$ to $-\frac{i}{\log x}$, continuing with the semicircle $\{  \frac{e^{i\theta}}{\log x} \ : \ -\frac{\pi}{2} \leq \theta \leq \frac{\pi}{2} \}$ , and ending with the line segment from $\frac{i}{\log x}$ to $-\frac{\sqrt{a}}{\log x} + \frac{i}{\log x}.$

Next we show that the contribution of the integration over $\mathcal{C}_1$ and $\mathcal{C}_2$ is small. When $w_4$ is on $\mathcal{C}_1$, by Lemma Theorem 6.7 in \cite{MV}, we obtain that for $z_1$ is on $K_3,$

\begin{align} \label{integral:overC1}
 \int_{\mathcal{C}_1} \frac{\zeta^{1/4}(1 + w_4)}{\zeta^{1/8}(1 + w_4 + z_1)} \frac{x^{w_4}}{w_4^2} \> dw_4 &\ll \int_{-\tfrac{\sqrt{a}}{\log x}}^{\tfrac{1}{\log x}} \frac{\left|\sigma + \frac{i}{\sqrt{\log x}} + \frac{1}{\log y}\right|^{1/8}}{\left|\sigma + \frac{i}{\sqrt{\log x}}\right|^{2 + 1/4}} x^{\sigma} \> d\sigma \\
&\ll  \int_{-\tfrac{\sqrt{a}}{\log x}}^{\tfrac{1}{\log x}}  \frac{1}{(\log x)^{1/16}}(\sqrt{\log x})^{2 + 1/4} x^{\sigma} \> d\sigma \nonumber \\
& \ll (\log x)^{1/16} \ll (\log q)^{1/16}\nonumber
\end{align}
 By Lemma \ref{lem:uppforL3K3}, we obtain that 
\begin{align} \label{integral:therestofC1}
& \frac{1}{(2\pi i)^{5} }\int_{K_{3}^{2}}\int_{L_{3}^{3}} \frac{\prod_{1 \leq i \leq 3} \zeta^{1/4}(1 + w_{i}) \zeta^{1/4}(1 + z_1 + z_{2})}{ \prod_{1 \leq i \leq 3} \zeta^{1/8}(1 + w_{i} + z_{1}) \zeta^{1/2}(1 + z_{2})}  \cdot  \frac{x^{w_1 + ... +w_{3}}y^{ z_1 +z_{2}}}{w_1^2...w_{3}^2z_1^3z_{2}^3}   \> dw_1...\>dw_3 \> d\vec{z} \\
&\ll (\log x)^{3 + 3/4}(\log y)^{4 - 5/8} \ll (\log q)^{7 + 1/8}. \nonumber
\end{align} 
Therefore the integral over $\mathcal{C}_1$ contributes $\ll (\log q)^{7 + 1/8 + 1/16} = o((\log q)^{8 + 1/4}).$

When $w_4$ is on $\mathcal{C}_2,$ by Theorem 6.7 in \cite{MV}, we obtain that for $z_1$ is on $K_3$
\begin{align} \label{integral:overC2}
  \int_{\mathcal{C}_2} \frac{\zeta^{1/4}(1 + w_1)}{\zeta^{1/8}(1 + w_1 + z_1)} \frac{x^{w_1}}{w_1^2} \> dw_1 &\ll \int_{\tfrac{1}{\log x}}^{\tfrac{1}{\sqrt{\log x}}} \frac{\left|-\frac{\sqrt{a}}{\log x} + i\tau + z_1\right|^{1/8}}{\left|-\frac{\sqrt{a}}{\log x} + i\tau \right|^{2 + 1/4}} e^{-\sqrt{a}}d\tau \\
&\ll e^{-\sqrt{a}} \int_{\tfrac{1}{\log x}}^{\tfrac{1}{\sqrt{\log x}}} \frac{1}{\tau^{2 + 1/8}}  +  \frac{|z_1|^{1/8}}{\tau^{2 + 1/4}} \> d\tau \nonumber \\
& \ll \frac{(\log x)^{9/8} +  |z_1|^{1/8}(\log x)^{5/4}}{e^{\sqrt{a}}}.\nonumber
\end{align}

Inserting the above bound to the full integral and applying Lemma \ref{lem:uppforL3K3}, we obtain that

\begin{align}  \label{integral:overC2conclusion}
&\int_{K_{3}^{2}}  \int_{L_{3}^{3}}\int_{\mathcal{C}_2} \left|\frac{\prod_{1 \leq i \leq 4} \zeta^{1/4}(1 + w_{i}) \zeta^{1/4}(1 + z_1 + z_{2})}{ \prod_{1 \leq i \leq 4} \zeta^{1/8}(1 + w_{i} + z_{1}) \zeta^{1/2}(1 + z_{2})}  \cdot  \frac{x^{w_1 + ... +w_{4}}y^{ z_1 +z_{2}}}{w_1^2...w_{4}^2z_1^3z_{2}^3}  \right| \> d|w_1|...\>d|w_4| \> d|z_1| \> d|z_2| \\
&\ll  e^{-\sqrt a} \int_{K_{3}^{2}} \int_{L_{3}^{3}}\left((\log x)^{9/8} +  |z_1|^{1/8}(\log x)^{5/4}\right) \nonumber\\
& \hspace{1.3 in} \cdot\left|\frac{\prod_{1 \leq i \leq 3} \zeta^{1/4}(1 + w_{i}) \zeta^{1/4}(1 + z_1 + z_{2})}{ \prod_{1 \leq i \leq 3} \zeta^{1/8}(1 + w_{i} + z_{1}) \zeta^{1/2}(1 + z_{2})}  \cdot  \frac{x^{w_1 + ... +w_{3}}y^{ z_1 +z_{2}}}{w_1^2...w_{4}^2z_1^3z_{2}^3}  \right| \> d|w_1|...\>d|w_4| \> d|z_1| \> d|z_2|\nonumber \\
&\ll \frac{(\log x)^{5}(\log y)^{3 + 1/4}}{e^{\sqrt a}}. \nonumber
\end{align}  

We now consider integration of $w_4$ over $\mathcal{C}_3.$ By Theorem 6.7 in \cite{MV}, we have $$\zeta^{1/4}(1 + w_4) = w_4^{-1/4}\left( 1 + O(|w_4|)\right),$$ and $$\zeta^{-1/8}(1 + w_4 + z_1) = (w_4 + z_1)^{1/8}\left(1 + O(|w_4| + |z_1|)\right). $$
Since $|w_4| \ll \frac{1}{\log x}$ and $|z_1| \ll \frac{1}{\sqrt{\log y}},$ by Lemma \ref{lem:uppforL3K3} the integration of the big-$O$ term contributes $o((\log q)^{8 + 1/4}).$ The integration over the main term is
\begin{equation} \label{eqn:maintermC3}
\frac{1}{2\pi i}\int_{\mathcal{C}_3} \frac{(w_4 + z_1)^{1/8}}{w_4^{2 + 1/4}} x^{w_4} \> dw_4 = \frac{z^{1/8}}{2\pi i}\int_{\mathcal{C}_3} \left(1 + \frac{w_4}{z_1}\right)^{1/8}\frac{x^{w_4}}{w_4^{2 + 1/4}} \> dw_4.
\end{equation}
Since $z_1$ is on $K_3,$ and $w_4$ is on $\mathcal{C}_3$, $|z_1| \geq \frac{1}{\log y} = \frac{a}{\log x},$ and $|w_4| \leq \frac{\sqrt{a + 1}}{\log x}.$ Hence $$\left|\frac{w_4}{z_1}\right| \leq \frac{\sqrt{a +1}}{a} \leq \sqrt{\frac{2}{a}}.$$
Since $a$ is chosen to be much larger than 2, we apply Taylor expansion to write 
$$ \left( 1 + \frac{w_4}{z_1}\right)^{1/8} = 1 + \sum_{n = 1}^{\infty}  a_n \left(\frac{w_4}{z_1}\right)^n,$$
where 
\begin{align*}
|a_n| = \left|\frac{\left(\tfrac{1}{8}\right)....\left(\tfrac{1}{8} - n - 1\right)}{n!}\right| = \frac{\left(\tfrac{1}{8}\right)}{1}\frac{\left(1 - \tfrac{1}{8}\right)}{2} ... \frac{\left|n - 1 - \tfrac{1}{8}\right|}{n} \leq \frac{1}{8}.
\end{align*}
Therefore
$$ \left( 1 + \frac{w_4}{z_1}\right)^{1/8} = 1 + O\left(\frac{1}{8\sqrt{a}}\right).$$
We will show that the contribution from the term $O\left(\frac{1}{8\sqrt{a}}\right)$ is small. Over $\mathcal{C}_3,$ we have
\begin{align} \label{integral:overC3w1}
\int_{\mathcal{C}_3} \left|\frac{x^{w_4}}{w_4^{2 + 1/4}}\right| \> d|w_4| &\ll
\int_{-\frac{\sqrt{a}}{\log x}}^{0} \frac{x^{\sigma}}{\left|\sigma + \frac{i}{\log x}\right|^{2 + 1/4}} \> d\sigma +  (\log x)^{1 + 1/4}  \\
&\ll (\log x)^{2 + 1/4} \int_{-\infty}^{0} x^{\sigma} \> d\sigma + (\log x)^{1 + 1/4} \ll (\log x)^{1 + 1/4}. \nonumber
\end{align}
By Lemma \ref{lem:uppforL3K3}, we obtain that
\begin{align}  \label{integral:overC3error}
& \int_{K_{3}^{2}}  \int_{L_{3}^{3}} |z_1|^{1/8}\left|\frac{\prod_{1 \leq i \leq 3} \zeta^{1/4}(1 + w_{i}) \zeta^{1/4}(1 + z_1 + z_{2})}{ \prod_{1 \leq i \leq 3} \zeta^{1/8}(1 + w_{i} + z_{1}) \zeta^{1/2}(1 + z_{2})}  \cdot  \frac{x^{w_1 + ... +w_{3}}y^{ z_1 +z_{2}}}{w_1^2...w_{3}^2z_1^3z_{2}^3}  \right| \> dw_3 \> dw_2 \> dw_1 \> d\vec{z} \\
&\ll (\log x)^{3 + 3/4}(\log y)^{3 + 1/4}. \nonumber
\end{align}  
From (\ref{eqn:maintermC3}), (\ref{integral:overC3w1}) and (\ref{integral:overC3error}), the contribution from the term $O\left(\frac{1}{8\sqrt{a}}\right)$ is $O\left( \frac{(\log x)^5(\log y)^{3 + 1/4}}{\sqrt{a}}\right).$  Hence we can write (\ref{M1for12}) as 
\begin{align} \label{M1atC3}
& \frac{1}{(2\pi i)^6}\int_{K_{3}^{2}} \int_{L_{3}^{3}}z_1^{1/8}\frac{\prod_{1 \leq i \leq 3} \zeta^{1/4}(1 + w_{i}) \zeta^{1/4}(1 + z_1 + z_{2})}{ \prod_{1 \leq i \leq 3} \zeta^{1/8}(1 + w_{i} + z_{1}) \zeta^{1/2}(1 + z_{2})}  \cdot  \frac{x^{ w_1 + ... +w_{3}}y^{ z_1 +z_{2}}}{w_1^2...w_{3}^2z_1^3z_{2}^3}  \int_{\mathcal{C}_3} \frac{x^{w_4}}{w_4^{2 + 1/4}} \> d\vec{w} \> d\vec{z} \\
&\hspace{2in}+ O\left( (\log x)^{5}(\log y)^{3 + 1/4}\left( \frac{1}{e^{\sqrt a}} + \frac{1}{\sqrt{a}}\right)\right). \nonumber
\end{align}
By the change of variables from $w_4$ to $w_4/\log x$, we see that the integration over $\mathcal{C}_3$ becomes
$$ \frac{(\log x)^{1 + 1/4}}{2\pi i} \int_{\mathcal{H}_3} w_4^{-2 - 1/4}e^{w_4} \> dw_4,$$
where $\mathcal{H}_3$ starts at $-\sqrt{a} - i$, loops around 0 and ends at $-\sqrt{a} + i.$ Let $\mathcal{H}$ be a contour consisting of line segments from $- \infty - i$ to $-\sqrt{a} - i$ and from $-\sqrt{a} + i$ to $-\infty + i.$ From Hankel's formula (e.g. C.3 in \cite{MV})
$$ \frac{1}{2\pi i} \int_{\mathcal{H}_3 \cup \mathcal{H}} w_4^{-2 - 1/4}e^{w_4} \> dw_4 =  \frac{1}{\Gamma(2 + 1/4)}.$$
The integral over $\mathcal{H} \ll \int_{-\infty}^{-\sqrt{a}} e^{x} \> dx \ll e^{-\sqrt{a}}.$ From (\ref{integral:overC3error}) and Hankel's formula, (\ref{M1atC3}) is 
\begin{align*} 
& \frac{1}{(2\pi i)^5}\frac{(\log x)^{1 + 1/4}}{\Gamma\left(2 + \frac{1}{4} \right)}\int_{K_{3}^{2}}\int_{L_{3}^{3}} \frac{\prod_{1 \leq i \leq 3} \zeta^{1/4}(1 + w_{i}) \zeta^{1/4}(1 + z_1 + z_{2})}{ \prod_{1 \leq i \leq 3} \zeta^{1/8}(1 + w_{i} + z_{1}) \zeta^{1/2}(1 + z_{2})}  \cdot  \frac{x^{ w_1 + ... +w_{3}}y^{ z_1 +z_{2}}}{w_1^2...w_{3}^2z_1^{3 - 1/8}z_{2}^3}  \> dw_3 \> dw_2 \> dw_1 \> d\vec{z} \\
&\hspace{2in}+ O\left( (\log x)^{5}(\log y)^{3 + 1/4}\left( \frac{1}{e^{\sqrt a}} + \frac{1}{\sqrt{a}}\right)\right). \nonumber
\end{align*}
We repeat the process of integrating $w_4$ to $w_1, w_2,$ and $w_3$ and obtain that $(\ref{M1for12})$ is 
\begin{align*} 
& \frac{1}{(2\pi i)^2}\frac{(\log x)^{5}}{\Gamma^4\left(2 + \frac{1}{4} \right)}\int_{K_{3}^{2}} \frac{ \zeta^{1/4}(1 + z_1 + z_{2})}{ \zeta^{1/2}(1 + z_{2})}  \cdot  \frac{y^{ z_1 +z_{2}}}{z_1^{3 - 1/2}z_{2}^3} \> d\vec{z} + O\left( (\log x)^{5}(\log y)^{3 + 1/4}\left( \frac{1}{e^{\sqrt a}} + \frac{1}{\sqrt{a}}\right)\right). \nonumber
\end{align*}
On $K_3,$ we have $\zeta^{-1/2}(z_2) = z_2^{1/2}\left( 1 + O(|z_2|)\right).$ The integration over the error term is $o((\log q)^{8 + 1/4}).$ Therefore the integral above is
 \begin{align*} 
& \frac{1}{(2\pi i)^2}\frac{(\log x)^{5}}{\Gamma^4\left(2 + \frac{1}{4} \right)}\int_{K_{3}^{2}}  \zeta^{1/4}(1 + z_1 + z_{2})  \frac{y^{ z_1 +z_{2}}}{z_1^{3 - 1/2}z_{2}^{3-1/2}} \> d\vec{z} + O\left( (\log x)^{5}(\log y)^{3 + 1/4}\left( \frac{1}{e^{\sqrt a}} + \frac{1}{\sqrt{a}}\right)\right). \nonumber
\end{align*}
By the same truncation arguments as in the beginning of the section, we can show that
\begin{align*}
\int_{K_{3}^{2}}  \zeta^{1/4}(1 + z_1 + z_{2})  \frac{y^{ z_1 +z_{2}}}{z_1^{3 - 1/2}z_{2}^{3-1/2}} \> d\vec{z} = \int_{(1/\log y)^{2}}  \zeta^{1/4}(1 + z_1 + z_{2})  \frac{y^{ z_1 +z_{2}}}{z_1^{3 - 1/2}z_{2}^{3-1/2}} \> d\vec{z} + o\left((\log q)^{3 + 1/4} \right). 
\end{align*}
Therefore (\ref{M1for12}) is 
\begin{align*} 
& \frac{1}{(2\pi i)^2}\frac{(\log x)^{5}}{\Gamma^4\left(2 + \frac{1}{4} \right)}\int_{(1/\log y)^{2}}  \zeta^{1/4}(1 + z_1 + z_{2})  \frac{y^{ z_1 +z_{2}}}{z_1^{3 - 1/2}z_{2}^{3-1/2}} \> d\vec{z} + O\left( (\log x)^{5}(\log y)^{3 + 1/4}\left( \frac{1}{e^{\sqrt a}} + \frac{1}{\sqrt{a}}\right)\right). 
\end{align*}

The power of $\zeta$ left in the integral is positive, and we apply Lemma \ref{lem:lowerboundz} to bound the main term above from below by $(\log x)^{5}(\log y)^{3 + 1/4}.$ Hence the above expression is 
$$ \geq \left(k_1 - \frac{k_2}{\sqrt{a}}\right)(\log x)^{5}(\log y)^{3 + 1/4},$$
where $k_1, k_2$ are absolute constants not depending on $a.$ Therefore for $a$ sufficiently large, and $k=1/2$, $I_l$ is bounded below by $(\log x)^{5}(\log y)^{3 + 1/4}$ as desired. 

\subsection{Modifications for the general case}
The proof of the general case is very similar with a few superficial differences.  For notational convenience, let 
$$
\mathcal Q = (\log x)^{1+2s}(\log y)^{2(2s-2r)+\frac{(s-r)^2}{s^2} - \frac{2(s-r)}{s}}
$$
As in \S \ref{sec:1/2case}, we first shift contours for $w_i$ for $1\leq i\leq 2s$.  Let us first shift $w_{2s}$.  When we shift, we replace $K_3$ by a path consisting of $\mathcal{C}_1, \mathcal{C}_2,$ and $\mathcal{C}_3$.  We remind the reader that the contribution of $\mathcal{C}_1$ is small since $w_{2s}$ is bounded away from $0$ so we save a power of $\log x$ from the $\frac{1}{w_{2s}^2}$ term.  The contribution from $\mathcal C_2$ is small since $|x^{w_{2s}}| = e^{-\sqrt{a}}$ on $\mathcal C_2$, so that the total contribution is $O \bfrac{Q} {e^{\sqrt{a}}}.$
The main term comes from $\mathcal C_3$ as before.  To be specific, we are now interested in the integral
\begin{eqnarray}
\label{eqn:maintermC3general}
&&\frac{1}{2\pi i}\int_{\mathcal{C}_3} \frac{\prod_{1\leq j\leq s-r}(w_{2s} + z_j)^{1/2s^2}}{w_{2s}^{2 + 1/2s}} x^{w_{2s}} \> dw_{2s} \\
&=& \prod_{1\leq j\leq s-r}z_j^{1/2s^2}\frac{1}{2\pi i}\int_{\mathcal{C}_3} \prod_{1\leq j\leq s-r}\left(1 + \frac{w_{2s}}{z_j}\right)^{1/2s^2}\frac{x^{w_{2s}}}{w_{2s}^{2 + 1/2s}} \> dw_{2s} \nonumber.
\end{eqnarray}  We have constructed this contour so that we can Taylor expand $\left(1 + \frac{w_{2s}}{z_j}\right)^{1/2s^2}$ to see that  
$$ \left( 1 + \frac{w_{2s}}{z_j}\right)^{1/2s^2} = 1 + O\left(\frac{1}{2s^2\sqrt{a}}\right).$$
An appeal to Lemma \ref{lem:uppforL3K3}, and the same calculation as before gives that the $O\left(\frac{1}{2s^2\sqrt{a}}\right)$ terms contribute $O\bfrac{(s-r)Q}{2s^2\sqrt{a}}$.  The rest of the integral is
$$\prod_{1\leq j\leq s-r}z_j^{1/2s^2}\frac{1}{2\pi i}\int_{\mathcal{C}_3} \frac{x^{w_{2s}}}{w_{2s}^{2 + 1/2s}} \> dw_{2s}
 = \frac{(\log x)^{1+\frac{1}{2s}}}{\Gamma(2+\frac{1}{2s})} \prod_{1\leq j\leq s-r}z_j^{1/2s^2} \left(1+O\bfrac{1}{e^{\sqrt{a}}}\right).
$$
Repeating this process for all the $w_i$ reduces the integral to a matter of studying 
$$\int_{K_3^{2s-2r}} \prod_{1\leq j, k\leq s-r} \zeta^{1/s^2}(1+z_j+z_{s-r+k}) \prod_{1\leq j\leq s-r} \frac{y^{z_j}}{z_j^{3-1/s}}\prod_{1\leq k\leq s-r} \frac{y^{z_{s-r+k}}}{z_{s-r+k}^{3-1/s}}d\vec z.
$$This can be bounded using Lemma \ref{lem:lowerboundz}, and the result follows.
\hskip 2 in

\paragraph{\bf Acknowledgements:} We would like to thank Professor Soundararajan for suggesting this problem, for his guidance, and for his constant encouragement.

\end{document}